\documentclass[11pt]{article}

\usepackage{amssymb,amsmath,amsfonts,amsthm}
\usepackage{latexsym}
\usepackage{graphics}
\usepackage{indentfirst}
\usepackage{hyperref}
\usepackage{comment}
\usepackage[shortlabels]{enumitem}
\usepackage[english]{babel}
\usepackage{tikz}
\usepackage{esdiff}

\usetikzlibrary{arrows.meta}

\newtheorem*{thm*}{Theorem}
\newtheorem{thm}{Theorem}
\newtheorem{lem}[thm]{Lemma}
\newtheorem{pro}[thm]{Proposition}

\newtheorem{cor}[thm]{Corollary}
\newtheorem{conj}[thm]{Conjecture}

\newtheorem{ques}[thm]{Question}
\newtheorem{defn}[thm]{Definition}

\newcommand{\HH}{\mathcal{H}}
\newcommand{\II}{\mathcal{I}}

\newcommand{\DD}{\mathcal{D}}
\newcommand{\N}{\mathbb{N}}

\allowdisplaybreaks

\begin{document}

\title{The DP Color Function of Clique-Gluings of Graphs}

\author{Hemanshu Kaul$^1$, Michael Maxfield$^2$, Jeffrey A. Mudrock$^3$, and Seth Thomason$^2$}

\footnotetext[1]{Department of Applied Mathematics, Illinois Institute of Technology, Chicago, IL 60616. E-mail: {\tt {kaul@iit.edu}}}

\footnotetext[2]{Department of Mathematics, College of Lake County, Grayslake, IL 60030.}

\footnotetext[3]{Department of Mathematics and Statistics, University of South Alabama, Mobile, AL 36688.  E-mail:  {\tt {mudrock@southalabama.edu}}}

\maketitle

\begin{abstract}
DP-coloring (also called correspondence coloring) is a generalization of list coloring that has been widely studied in recent years after its introduction by Dvo\v{r}\'{a}k and Postle in 2015. As the analogue of the chromatic polynomial of a graph $G$, $P(G,m)$, the DP color function of $G$, denoted by $P_{DP}(G,m)$, counts the minimum number of DP-colorings over all possible $m$-fold covers. Formulas for chromatic polynomials of clique-gluings of graphs, a fundamental graph operation, are well-known, but the effect of such gluings on the DP color function is not well understood. In this paper we study the DP color function of $K_p$-gluings of graphs. Recently, Becker et. al. asked whether $P_{DP}(G,m) \leq (\prod_{i=1}^n P_{DP}(G_i,m))/\left( \prod_{i=0}^{p-1} (m-i) \right)^{n-1}$ whenever $m \geq p$, where the expression on the right is the DP-coloring analogue of the corresponding chromatic polynomial formula for a $K_p$-gluing, $G$, of $G_1, \ldots, G_n$.  Becker et. al. showed this inequality holds when $p=1$. In this paper we show this inequality holds for edge-gluings ($p=2$). On the other hand, we show it does not hold for triangle-gluings ($p=3$), which also answers a question of Dong and Yang (2021). Finally, we show a relaxed version, based on a class of $m$-fold covers that we conjecture would yield the fewest DP-colorings for a given graph, of the inequality holds when $p \geq 3$. 

\medskip

\noindent {\bf Keywords.} DP-coloring, correspondence coloring, chromatic polynomial, DP color function, clique-gluing, clique-sum.

\noindent \textbf{Mathematics Subject Classification.} 05C15, 05C30, 05C69

\end{abstract}

\section{Introduction}\label{intro}

In this paper all graphs are nonempty, finite, simple graphs unless otherwise noted.  Generally speaking we follow West~\cite{W01} for terminology and notation.  The set of natural numbers is $\N = \{1,2,3, \ldots \}$.  For $m \in \N$, we write $[m]$ for the set $\{1, \ldots, m \}$.  Given a set $A$, $\mathcal{P}(A)$ is the power set of $A$.  If $G$ is a graph and $S, U \subseteq V(G)$, we use $G[S]$ for the subgraph of $G$ induced by $S$, and we use $E_G(S, U)$ for the set consisting of all the edges in $E(G)$ that have one endpoint in $S$ and the other in $U$.  If $u$ and $v$ are adjacent in $G$, $uv$ or $vu$ refers to the edge between $u$ and $v$.  For $v \in V(G)$, we write $d_G(v)$ for the degree of vertex $v$ in the graph $G$ and $\Delta(G)$ for the maximum degree of a vertex in $G$.  We write $N_G(v)$ (resp. $N_G[v]$) for the neighborhood (resp. closed neighborhood) of vertex $v$ in the graph $G$.  We use $\omega(G)$ to denote the clique number of the graph $G$.  If $e \in E(G)$, we write $G \cdot e$ for the graph obtained from $G$ by contracting the edge $e$.  When $C$ is a cycle on $n$ vertices ($n \geq 3$ since $C$ is simple), $V(C)= \{v_1, \ldots, v_n \}$, and $E(C) = \{\{v_1,v_2 \}, \{v_2, v_3 \}, \ldots, \{v_{n-1}, v_n \}, \{v_n, v_1 \} \}$, then we say the vertices of $C$ are written in \emph{cyclic order} when we write $v_1, \ldots, v_n$.  If $G$ and $H$ are vertex disjoint graphs, we write $G \vee H$ for the join of $G$ and $H$. When $G = K_1$, $G \vee H$ is the \emph{cone} of $H$, and the vertex in $V(G)$ is called the \emph{universal vertex} of $G \vee H$.  

\subsection{List Coloring and DP-Coloring} \label{basic}

In the classical vertex coloring problem we wish to color the vertices of a graph $G$ with up to $m$ colors from $[m]$ so that adjacent vertices receive different colors, a so-called \emph{proper $m$-coloring}.  List coloring is a variation on classical vertex coloring that was introduced independently by Vizing~\cite{V76} and Erd\H{o}s, Rubin, and Taylor~\cite{ET79} in the 1970s.  For list coloring, we associate a \emph{list assignment} $L$ with a graph $G$ such that each vertex $v \in V(G)$ is assigned a list of available colors $L(v)$ (we say $L$ is a list assignment for $G$).  We say $G$ is \emph{$L$-colorable} if there is a proper coloring $f$ of $G$ such that $f(v) \in L(v)$ for each $v \in V(G)$ (we refer to $f$ as a \emph{proper $L$-coloring} of $G$).  A list assignment $L$ is called a \emph{$k$-assignment} for $G$ if $|L(v)|=k$ for each $v \in V(G)$.  We say $G$ is \emph{$k$-choosable} if $G$ is $L$-colorable whenever $L$ is a $k$-assignment for $G$. 

In 2015, Dvo\v{r}\'{a}k and Postle~\cite{DP15} introduced a generalization of list coloring called DP-coloring (they called it correspondence coloring) in order to prove that every planar graph without cycles of lengths 4 to 8 is 3-choosable.  Intuitively, DP-coloring is a variation on list coloring where each vertex in the graph still gets a list of colors, but identification of which colors are different can change from edge to edge.  Following~\cite{BK17}, we now give the formal definition. Suppose $G$ is a graph.  A \emph{cover} of $G$ is a pair $\mathcal{H} = (L,H)$ consisting of a graph $H$ and a function $L: V(G) \rightarrow \mathcal{P}(V(H))$ satisfying the following four requirements:

\vspace{5mm}

\noindent(1) the set $\{L(u) : u \in V(G) \}$ is a partition of $V(H)$ of size $|V(G)|$; \\
(2) for every $u \in V(G)$, the graph $H[L(u)]$ is complete; \\
(3) if $E_H(L(u),L(v))$ is nonempty, then $u=v$ or $uv \in E(G)$; \\
(4) if $uv \in E(G)$, then $E_H(L(u),L(v))$ is a matching (the matching may be empty).

\vspace{5mm}

Suppose $\mathcal{H} = (L,H)$ is a cover of $G$.  We refer to the edges of $H$ connecting distinct parts of the partition $\{L(v) : v \in V(G) \}$ as \emph{cross-edges}.  An \emph{$\mathcal{H}$-coloring} of $G$ is an independent set in $H$ of size $|V(G)|$.  It is immediately clear that an independent set $I \subseteq V(H)$ is an $\mathcal{H}$-coloring of $G$ if and only if $|I \cap L(u)|=1$ for each $u \in V(G)$.  We say $\mathcal{H}$ is \emph{$m$-fold} if $|L(u)|=m$ for each $u \in V(G)$~\footnote{Throughout the following, for any $u\in V(G)$ we always label the vertices in $L(u)$ by $(u, 1), (u, 2), \ldots, (u, m)$.}. An $m$-fold cover $\mathcal{H}$ is a \emph{full cover} if for each $uv \in E(G)$, the matching $E_{H}(L(u),L(v))$ is perfect.  The \emph{DP-chromatic number} of $G$, $\chi_{DP}(G)$, is the smallest $m \in \N$ such that $G$ has an $\mathcal{H}$-coloring whenever $\mathcal{H}$ is an $m$-fold cover of $G$. 

Suppose $\mathcal{H} = (L,H)$ is an $m$-fold cover of $G$. 
Suppose $U \subseteq V(G)$. Let $\mathcal{H}_{U} = (L_{U},H_{U})$ where $L_{U}$ is the restriction of $L$ to $U$ and $H_{U} = H[\bigcup_{u \in U} L(u)]$. Clearly, $\mathcal{H}_{U}$ is an $m$-fold cover of $G[U]$. We call $\mathcal{H}_{U}$ the \emph{subcover of }$\mathcal{H}$\emph{ induced by }$U$.
Suppose $G'$ is a subgraph of $G$. Let $\mathcal{H}_{G'} = (L_{G'},H_{G'})$ where $L_{G'}$ is the restriction of $L$ to $V(G')$ and $H_{G'}=H[\bigcup_{u \in V(G')} L(u)] - \bigcup_{uv \in E(G) - E(G')}E_H(L(u),L(v))$. Clearly, $\mathcal{H}_{G'}$ is an $m$-fold cover of $G'$. We call $\mathcal{H}_{G'}$ the \emph{subcover of }$\mathcal{H}$\emph{ corresponding to }$G'$.

Suppose $\mathcal{H} = (L,H)$ is an $m$-fold cover of $G$.  We say that $\mathcal{H}$ has a \emph{canonical labeling} if it is possible to name the vertices of $H$ so that $L(u) = \{ (u,j) : j \in [m] \}$ and $(u,j)(v,j) \in E(H)$ for each $j \in [m]$ whenever $uv \in E(G)$.~\footnote{When $\mathcal{H}=(L,H)$ has a canonical labeling, we will always refer to the vertices of $H$ using this naming scheme.}  Now, suppose $\mathcal{H}$ has a canonical labeling and $G$ has a proper $m$-coloring.  Then, if $\mathcal{I}$ is the set of $\mathcal{H}$-colorings of $G$ and $\mathcal{C}$ is the set of proper $m$-colorings of $G$, the function $f: \mathcal{C} \rightarrow \mathcal{I}$ given by $f(c) = \{ (v, c(v)) : v \in V(G) \}$ is a bijection.  Also, given an $m$-assignment $L$ for a graph $G$, it is easy to construct an $m$-fold cover $\mathcal{H}'$ of $G$ such that $G$ has an $\mathcal{H}'$-coloring if and only if $G$ has a proper $L$-coloring (see~\cite{BK17}). In fact, there is a one-to-one correspondence between proper $L$-colorings of $G$ and the $\mathcal{H}'$-colorings of $G$. So, it is natural to consider the problem of counting the DP-colorings of a graph $G$ as a generalization of the chromatic polynomial of $G$ and the counting of list colorings of $G$.
 
\subsection{Counting Proper Colorings and List Colorings}

In 1912 Birkhoff introduced the notion of the chromatic polynomial with the hope of using it to make progress on the four color problem.  For $m \in \N$, the \emph{chromatic polynomial} of a graph $G$, $P(G,m)$, is the number of proper $m$-colorings of $G$. It is well-known that $P(G,m)$ is a polynomial in $m$ of degree $|V(G)|$ (see~\cite{B12}).   For example, $P(K_n,m) = \prod_{i=0}^{n-1} (m-i)$, $P(C_n,m) = (m-1)^n + (-1)^n (m-1)$, and $P(T,m) = m(m-1)^{n-1}$ whenever $T$ is a tree on $n$ vertices (see~\cite{W01}). 

We now mention a chromatic polynomial formula that will be important in this paper.  Suppose that $n \geq 2$, $G_1, \ldots, G_n$ are vertex disjoint graphs, and $1 \le p \le \min_{i \in [n]}\{\omega(G_i)\}$.  Choose a copy of $K_p$ contained in each $G_i$ and form a new graph $G$, called a \emph{$K_p$-gluing of $G_1, \ldots, G_n$}, from the union of $G_1, \ldots, G_n$ by arbitrarily identifying the chosen copies of $K_p$; that is, if $\{u_{i,1}, \ldots, u_{i,p} \}$ is the vertex set of the chosen copy of $K_p$ in $G_i$ for each $i \in [n]$, then identify the vertices $u_{1,j}, \ldots, u_{n,j}$ as a single vertex $u_j$ for each $j \in [p]$.~\footnote{This is equivalent to the well-studied notion of \emph{clique-sum} of  $G_1, \ldots, G_n$ where no edges are removed after the identification of the cliques.\label{clique-sum}} When $p=1$ this is called \emph{vertex-gluing}, when $p=2$ this is called \emph{edge-gluing} (see~\cite{DKT05}), and when $p=3$ we will call it \emph{triangle-gluing}. Let $\bigoplus_{i=1}^{n}(G_i,p)$ denote the family of all $K_p$-gluings of $G_1, \ldots, G_n$. It is well-known that for any $G\in \bigoplus_{i=1}^{n}(G_i,p)$, $P(G,m) = \prod_{i=1}^n P(G_i,m)/ \left( \prod_{i=0}^{p-1} (m-i) \right)^{n-1}$ whenever $m \geq p$ (see~\cite{B94, DKT05}).

Clique-gluing and the closely related clique-sum are fundamental graph operations which have been used to give a structural characterization of many families of graphs (see the discussion and examples in~\cite{KT90}). A simple example is that chordal graphs are precisely the graphs that can be formed by clique-gluings of cliques. While the most famous example would be Robertson and Seymour's seminal Graph Minor Structure Theorem  (\cite{RS03}) characterizing minor-free families of graphs.

The notion of chromatic polynomial was extended to list coloring in the early 1990s~\cite{AS90}. If $L$ is a list assignment for $G$, we use $P(G,L)$ to denote the number of proper $L$-colorings of $G$. The \emph{list color function} $P_\ell(G,m)$ is the minimum value of $P(G,L)$ where the minimum is taken over all possible $m$-assignments $L$ for $G$.  Clearly, $P_\ell(G,m) \leq P(G,m)$ for each $m \in \N$.  In general, the list color function can differ significantly from the chromatic polynomial for small values of $m$.  However, for large values of $m$,  Dong and Zhang~\cite{DZ22} (improving upon results in~\cite{D92}, \cite{T09}, and~\cite{WQ17}) showed that for any graph $G$ with at least 4 edges, $P_{\ell}(G,m)=P(G,m)$ whenever $m \geq |E(G)|-1$.  It is also known that $P_{\ell}(G,m)=P(G,m)$ for all $m \in \N$ when $G$ is a cycle or chordal (see~\cite{KN16, AS90}). See~\cite{KK22, T09} for a survey of known results and open questions on the list color function.

\subsection{The DP Color Function and Motivating Question}

Two of the current authors (Kaul and Mudrock in~\cite{KM19}) introduced a DP-coloring analogue of the chromatic polynomial to gain a better understanding of DP-coloring and use it as a tool for making progress on some open questions related to the list color function.  Since its introduction in 2019, the DP color function has received some attention in the literature (see~\cite{BH21, BK21, DY21, DY22, HK21, KM19, KM21, M21, MT20}).  Suppose $\mathcal{H} = (L,H)$ is a cover of graph $G$.  Let $P_{DP}(G, \mathcal{H})$ be the number of $\mathcal{H}$-colorings of $G$.  Then, the \emph{DP color function} of $G$, $P_{DP}(G,m)$, is the minimum value of $P_{DP}(G, \mathcal{H})$ where the minimum is taken over all possible $m$-fold covers $\mathcal{H}$ of $G$.~\footnote{We take $\N$ to be the domain of the DP color function of any graph.} It is easy to show that for any graph $G$ and $m \in \N$, $P_{DP}(G, m) \leq P_\ell(G,m) \leq P(G,m)$.  Note that if $G$ is a disconnected graph with components: $H_1, H_2, \ldots, H_t$, then $P_{DP}(G, m) = \prod_{i=1}^t P_{DP}(H_i,m)$.  So, we will only consider connected graphs from this point forward unless otherwise noted.

The list color function and DP color function of certain graphs behave similarly.  However, for some graphs there are surprising differences.  For example, similar to the list color function,  $P_{DP}(G,m) = P(G,m)$ for every $m \in \N$  whenever $G$ is chordal or an odd cycle.  On the other hand, $P_{DP}(C_{2k+2},m)=(m-1)^{2k+2}-1 < P(C_{2k+2},m)$ whenever $m \geq 2$ (see~\cite{KM19})~\footnote{These results on the DP color functions of cycles and chordal graphs will be useful to keep in mind for this paper.}.  Even more dramatically, the following result recently appeared in the literature.

\begin{thm} [\cite{DY21}] \label{thm: evengirth}
For graph $G$ let $\ell_G : E(G) \rightarrow \N \cup \{\infty\}$ be the function that maps each cut-edge in $G$ to $\infty$ and maps each non-cut-edge $e \in E(G)$ to the length of a shortest cycle in $G$ containing $e$.  If $G$ contains an edge $l$ such that $\ell_G(l)$ is even, then there exists $N \in \N$ such that $P_{DP}(G,m) < P(G,m)$ whenever $m \geq N$.  
\end{thm}

This result is particularly interesting since we know the list color function of any graph eventually equals its chromatic polynomial.  With this in mind, our motivation for this paper began with a question presented in~\cite{BH21}.  More specifically, Theorem~\ref{thm: evengirth} tells us that the DP color function of a $K_p$-gluing of graphs may not eventually equal the chromatic polynomial of the resulting graph.  We may wonder however if the formula for the chromatic polynomial of a $K_p$-gluing of graphs (i.e., for any $G\in \bigoplus_{i=1}^{n}(G_i,p)$, $P(G,m) = \prod_{i=1}^n P(G_i,m)/ \left( \prod_{i=0}^{p-1} (m-i) \right)^{n-1}$ whenever $m \geq p$) has a DP-coloring analogue.

\begin{ques} [\cite{BH21}] \label{ques: classify}
If $p \ge 1$, $n \geq 2$, $G_1, \ldots, G_n$ are vertex disjoint graphs, and $G \in \bigoplus_{i=1}^{n}(G_i,p)$, is it the case that 
$$P_{DP}(G,m) \leq \frac{\prod_{i=1}^n P_{DP}(G_i,m)}{\left( \prod_{i=0}^{p-1} (m-i) \right)^{n-1}}$$ 
for all $m \geq p$? 
\end{ques} 

In~\cite{BH21} it is shown that the inequality in Question~\ref{ques: classify} need not be an equality in all situations: when $p=1$, $n=2$, $G_1 = G_2 = K_1 \vee C_4$, and $G$ is formed by gluing the universal vertices of $G_1$ and $G_2$, then $P_{DP}(G,4) < (P_{DP}(G_1,4)P_{DP}(G_2,4))/4$.  Furthermore, it is shown in~\cite{BH21} that the answer to Question~\ref{ques: classify} is yes when $p=1$.  It is worth mentioning that a simpler version of Question~\ref{ques: classify} has also recently appeared in the literature.  Recall that a vertex $v$ in a graph $G$ is called \emph{simplicial} if either $d_G(v)=0$ or $G[N_G(v)]$ is a complete graph.

\begin{ques} [\cite{DY21}] \label{ques: easier}
If $v$ is a simplicial vertex of $G$, is it true that for all positive integers $m \geq d_G(v)$, $P_{DP}(G,m) \leq (m- d_G(v)) P_{DP}(G - \{v\}, m)$?
\end{ques}

To see that Question~\ref{ques: easier} is a simpler version of Question~\ref{ques: classify}, first note that the inequality obviously holds when $d_G(v)=0$.  Then, suppose $v$ is a simplicial vertex of $G$ with positive degree, and let $t=d_G(v)$.  Then, $G$ is a $K_{t}$-gluing of $G - \{v\}$ and a copy of $K_{t+1}$.  Dong and Yang showed in~\cite{DY21} that Question~\ref{ques: easier} has a positive answer when $d_G(v) \in \{0,1,2\}$. In this paper we make further progress on Question~\ref{ques: classify}, and we show the answer to Question~\ref{ques: easier} is no when $d_G(v)=3$.  

\subsection{Outline of the Paper and a Conjecture}

We now present an outline of the paper.  In Section~\ref{edge} we show that the answer to Question~\ref{ques: classify} is yes for edge-gluings ($p=2$) which generalizes the result of Dong and Yang~\cite{DY21} that shows the answer to Question~\ref{ques: easier} is yes when $d_G(v)=2$.

\begin{thm}\label{thm: upperbound2}
Suppose that $G_1,\ldots,G_n$ are vertex disjoint graphs for some $n \geq 2$ with $u_iv_i \in E(G_i)$ for each $i \in [n]$. Suppose that $G$ is the graph obtained by identifying $u_1,\ldots,u_n$ as the same vertex $u$ and $v_1,\ldots,v_n$ as the same vertex $v$. Then \\ $P_{DP}(G,m) \leq \prod_{i=1}^{n} P_{DP}(G_i,m)/(m(m-1))^{n-1}$ whenever $m \geq 2$. 
\end{thm}

We end Section~\ref{edge} by showing how the techniques, including the relevant concepts of ``gluing'' and ''splitting'' covers of graphs, used to prove Theorem~\ref{thm: upperbound2} yield a new method for finding the formulas of the DP color functions of chorded cycles~\footnote{The DP color functions of all chorded cycles were determined in~\cite{KM19} using a different technique.}.

In Section~\ref{triangle} we show that the answer to both Questions~\ref{ques: classify} and~\ref{ques: easier} is no for $p=3$ and $d_G(v)=3$ respectively. We construct a counterexample with a graph based on tessellations of an appropriate number of triangles in $K_1  \vee \Theta(2,2,2)$ where $\Theta(2, 2, 2)$ consists of a pair of end vertices joined by three internally disjoint paths each of length two \footnote{More generally, a \emph{Theta graph} $\Theta(l_1, l_2, l_3)$ consists of a pair of end vertices joined by $3$ internally disjoint paths of lengths $l_1, l_2, l_3 \in \N$.}.  

While the answers to Questions~\ref{ques: classify} and~\ref{ques: easier} are unknown for each $p \geq 4$ and $d_G(v) \geq 4$ respectively, our result in Section~\ref{triangle} makes us suspect both questions have a negative answer for each of these values.  Nevertheless, in Section~\ref{general} we show that the answer to a relaxed version of Question~\ref{ques: classify} is yes.  Specifically, suppose $p \geq 1$, and $G$ is a graph such that $\{v_1,\ldots,v_p\}$ is a clique in $G$, and let $K= \{v_1 , \ldots , v_p\}$.  We say an $m$-fold cover $\HH = (L,H)$ of $G$ is \emph{conducive} to $\{v_1,\ldots,v_p\}$ if $\HH$ is full and the $m$-fold cover $\mathcal{H}_K$ of $G[K]$ admits a canonical labeling.  The \emph{$K$-canonical DP-Color Function of $G$}, $P_{DP}'(G,K,m)$, is the minimum value of $P_{DP}(G,\HH')$ where the minimum is taken over all $m$-fold covers $\HH'$ of $G$ conducive to $K$. Clearly, $P_{DP}'(G,K,m) \geq P_{DP}(G,m)$, and in fact, the construction in Section~\ref{triangle} yields the only example, of which we are aware, of a graph $G$ and clique $K$ in $G$ for which $P_{DP}'(G,K,m) > P_{DP}(G,m)$ for some $m \in \N$. So, the following result gives an affirmative answer to a relaxed version of Question~\ref{ques: classify}.

\begin{thm}\label{thm: upperboundFixed}
Suppose that $G_1,\ldots,G_n$ are vertex disjoint graphs for some $n \geq 2$ and $G \in \bigoplus_{i=1}^{n} (G_i,p)$ where for each $i \in [n]$, $K_i = \{u_{i,1},\ldots,u_{i,p}\}$ is a clique in $G_i$ and $G$ is obtained by identifying $u_{1,q},\ldots,u_{n,q}$ as the same vertex $u_q$ for each $q \in [p]$. Then $P_{DP}(G,m) \leq \prod_{i=1}^{n} P_{DP}'(G_i,K_i,m)/(\prod_{i=0}^{p-1} (m-i))^{n-1}$ whenever $m \geq p$.
\end{thm}

Suppose that $K$ is a clique in $G$.  It is not hard to see that when $|K| \in [2]$, $P_{DP}'(G,K,m) = P_{DP}(G,m)$ for all $m \in \N$.  So, Theorem~\ref{thm: upperboundFixed} implies Theorem~\ref{thm: upperbound2}.~\footnote{We still give a direct proof of Theorem~\ref{thm: upperbound2} as the concepts introduced in Section~\ref{edge} are of independent interest, easier to apply, and help motivate our study of conducive covers.}  As we already mentioned, our counterexample in Section~\ref{triangle} demonstrates that this equality need not hold when $|K|=3$, but only for a small value of $m$ (i.e., $m=4$). This motivates the following conjecture.

\begin{conj} \label{ques: nature}
Suppose that $K$ is a clique in a graph $G$ with $|K| \geq 3$.  Then, there is an $N \in \N$ such that $P_{DP}'(G,K,m) = P_{DP}(G,m)$ whenever $m \geq N$.  
\end{conj} 

If true, Conjecture~\ref{ques: nature} would yield an affirmative answer to both Questions~\ref{ques: classify} and~\ref{ques: easier} for sufficiently large $m$ by Theorem~\ref{thm: upperboundFixed}.  However, the truth of Conjecture~\ref{ques: nature} would actually say something quite deep about the nature of sufficiently large covers that yield the fewest colorings:  Given any graph $G$, clique $K$ in $G$, and sufficiently large $m$, there must be an $m$-fold cover $\HH$ of $G$ such that $P_{DP}(G, \HH) = P_{DP}(G,m)$ and $\HH$ is conducive to $K$.

\section{Edge-Gluings} \label{edge}

In this section we will prove Theorem~\ref{thm: upperbound2}.  We begin with some notation that will be used throughout the paper.  Whenever $\HH = (L,H)$ is an $m$-fold cover of $G$ and $P \subseteq V(H)$, we let $N(P,\HH)$ be the number of $\HH$-colorings containing $P$.  The following proposition will be of fundamental importance throughout the paper.

\begin{pro} [\cite{KM19}] \label{pro: tree}
Suppose $T$ is a tree and $\mathcal{H} = (L,H)$ is a full $m$-fold cover of $T$.  Then, $\mathcal{H}$ has a canonical labeling.
\end{pro}

We now present two important definitions that will be used in this Section.  We define separated covers, a natural analogue of ``splitting'' an $m$-fold cover of $G\in \bigoplus_{i=1}^{n}(G_i,p)$ into separate $m$-fold covers for each $G_i$ (this notion was first defined in~\cite{BH21}).

\begin{defn}\label{defn: separation}
For some $n \geq 2$, suppose that $G_1,\ldots,G_n$ are vertex disjoint graphs such that that $\{u_{i,1},\ldots,u_{i,p}\}$ is a clique in $G_i$ for each $i \in [n]$ and some $p \in \N$. Let $G$ be the graph obtained by identifying $u_{1,q},\ldots,u_{n,q}$ as the same vertex $u_q$ for each $q \in [p]$. Suppose $\HH = (L,H)$ is an arbitrary $m$-fold cover of $G$. For each $i \in [n]$, \textbf{the separated cover of $G_i$ obtained from $\HH$} is an $m$-fold cover $\HH_i = (L_i,H_i)$ of $G_i$ defined as follows. Assume that $L(x) = \{(x,j) : j \in [m]\}$ for each $x \in V(G)$. Let $L_i(x) = \{(x,j) : j \in [m]\}$ for each $x \in V(G_i)$. Construct edges of $H_i$ so that $H_i$ is isomorphic to $H_i' = H[(\bigcup_{x \in V(G_i) - \{u_{i,q} : q \in [p]\}} L(x)) \cup (\bigcup_{q=1}^{p} L(u_q))]$ and $f: V(H_i') \rightarrow V(H_i)$ given by
\[ f((x,j)) = \begin{cases}
    (x,j) & \text{if } x \in V(G_i) - \{u_{i,q} : q \in [p]\}\\
    (u_{i,q},j) & \text{if } x = u_q \text{ for some } q \in [p]
   \end{cases}
\]
is a graph isomorphism.
\end{defn}

Given vertex disjoint graphs $G_1, \ldots, G_n$, we define amalgamated cover, a natural analogue of ``gluing'' $m$-fold covers of each $G_i$ together so that we get an $m$-fold cover for some $G\in \bigoplus_{i=1}^{n}(G_i,2)$. In addition to specifying the graphs $G_i$, and the glued edge in each graph that leads to $G$, we also need an ordered list of permutations, $F=(f_2,\ldots, f_n)$. $F$ is simply a convenient device for keeping track of which edges in the covers of $G_i$ corresponding to the glued edge are identified in the amalgamated cover of $G$. A generalization of this definition and its illustration is given in Section~\ref{general}. 

\begin{defn}\label{defn: edge amalgamation}
For some $n \geq 2$, suppose that $G_1,\ldots,G_n$ are vertex disjoint graphs such that $u_iv_i \in E(G_i)$ for each $i \in [n]$. For each $i \in [n]$, suppose $\HH_i = (L_i,H_i)$ is an $m$-fold cover of $G_i$. Let $f_{k+1}$ be a permutation of $[m]$ for each $k \in [n-1]$, and let $F = (f_2,\ldots,f_n)$. Let $G$ be the graph obtained by identifying $u_1,\ldots,u_n$ as the same vertex $u$ and by identifying $v_1,\ldots,v_n$ as the same vertex $v$. \textbf{The $F$-amalgamated $m$-fold cover of $G$ obtained from $\HH_1,\ldots,\HH_n$} is an $m$-fold cover $\HH = (L,H)$ of $G$ defined as follows. In the special case where $n = 2$, we may also say $f_2$-amalgamated $m$-fold cover of $G$ obtained from $\HH_1$ and $\HH_2$.

For each $i \in [n]$, assume $L_i(x) = \{(x,j): j \in [m]\}$ for each $x \in V(G_i)$. For each $i \in [n]$, rename the vertices in $L_i(v_i)$ so that $(u_i,j)(v_i,j) \in E(H_i)$ for each $j \in [m]$.  Let $L(x) = \{(x,j) : j \in [m] \}$ for each $x \in V(G)$.  For each $i \in [n]$, let $X_i$ be the set of edges in $H_i$ that are incident with at least one element in $L_i(u_i) \cup L_i(v_i)$. Construct edges in $H$ so that $H[L(u)]$ and $H[L(v)]$ are cliques and $H$ contains the edges in $\bigcup_{i=1}^{n} (E(H_i) - X_i)$. Then for each $j \in [m]$ and $k \in [n-1]$, whenever $(x,r)(u_1,j) \in E(H_1)$ where $(x,r) \notin L_1(u_1) \cup L_1(v_1)$ or $(x,r)(u_{k+1},f_{k+1}(j)) \in E(H_{k+1})$ where $(x,r) \notin L_{k+1}(u_{k+1}) \cup L_{k+1}(v_{k+1})$, construct the edge $(x,r)(u,j)$ in $H$. For each $j \in [m]$ and $k \in [n-1]$, whenever $(x,r)(v_1,j) \in E(H_1)$ where $(x,r) \notin L_1(u_1) \cup L_1(v_1)$ or $(x,r)(v_{k+1},f_{k+1}(j)) \in E(H_{k+1})$ where $(x,r) \notin L_{k+1}(u_{k+1}) \cup L_{k+1}(v_{k+1})$, construct the edge $(x,r)(v,j)$ in $H$. Finally, for each $j \in [m]$, construct the edge $(u,j)(v,j)$ in $H$.
\end{defn}

Although the details of this definition are unavoidably technical, the underlying idea is simple. The $F$-amalgamated cover $\HH = (L,H)$ for $G$ is obtained by using $F$ to identify one edge from each of $E_{H_1}(L_1(u_1), L_1(v_1))$, $E_{H_2}(L_2(u_2), L_2(v_2))$,..., $E_{H_n}(L_n(u_n), L_n(v_n))$ as a single edge in $H$ a total of $m$ times. Then, the original covers are preserved on the portion of each graph that excludes the glued edge. 

Before proving Theorem~\ref{thm: upperbound2}, we need a lemma.

\begin{lem}\label{lem: upperGen2}
Suppose that $G_1$ and $G_2$ are vertex disjoint graphs. Suppose that $m \in \N$ and that $\HH_i = (L_i,H_i)$ is an $m$-fold cover for $G_i$ for each $i \in [2]$. Suppose that $u_iv_i \in E(G_i)$ for each $i \in [2]$. Suppose $f$ is a permutation of $[m]$. Let $G$ be the graph obtained by identifying $u_1$ and $u_2$ as the same vertex $u$ and identifying $v_1$ and $v_2$ as the same vertex $v$, $\HH$ be the $f$-amalgamated $m$-fold cover of $G$ obtained from $\HH_1$ and $\HH_2$, and $D = \sum_{j_2=1}^{m} \sum_{j_1=1}^{m} N(\{(u_1,j_1),(v_1,j_2)\},\HH_1) N(\{(u_2,f(j_1)),(v_2,f(j_2))\},\HH_2)$. Then $P_{DP}(G,\HH) = D$ which implies $P_{DP}(G,m) \leq D$.
\end{lem}
\begin{proof}
Note $P_{DP}(G,\HH) = \sum_{j=1}^{m}\sum_{i=1}^m N(\{(u,i),(v,j)\},\HH)$ and $N(\{(u,i),(v,j)\},\HH) = 0$ when $(u,i)(v,j) \in E(H)$. Let $P = \{(u,j_1),(v,j_2)\}$ where $j_1,j_2$ are fixed elements of $[m]$ and $(u,j_1)(v,j_2) \notin E(H)$. We will determine $N(P,\HH)$. Let $P_1 = \{(u_1,j_1),(v_1,j_2)\}$ and $P_2 = \{(u_1,f(j_1)),(v_1,f(j_2))\}$. Let $\II_1$ be the set of all $\HH_1$-colorings of $G_1$ that contain $P_1$, $\II_2$ be the set of all $\HH_2$-colorings of $G_2$ that contain $P_2$, and $\II$ be the set of all $\HH$-colorings of $G$ that contain $P$. Let $g: \II_1 \times \II_2 \rightarrow \II$ be the function given by $g((I_1,I_2)) = (I_1 - P_1) \cup (I_2 - P_2) \cup P$. It is easy to check that $(I_1 - P_1) \cup (I_2 - P_2) \cup P$ is an independent set of size $|V(G)|$ in $H$. Also, $g$ is a bijection. As such, $N(P,\HH) = |\II| = |\II_1||\II_2| = N(P_1,\HH_1)N(P_2,\HH_2)$. Therefore, $P_{DP}(G,\HH) = D$ which implies that $P_{DP}(G,m) \leq D$.
\end{proof}

We are now ready to prove Theorem~\ref{thm: upperbound2}.

\begin{proof}
The proof is by induction on $n$. We begin by proving the result for $n = 2$. Suppose $\HH_i = (L_i,H_i)$ is a full $m$-fold cover of $G_i$ such that $P_{DP}(G_i,\HH_i) = P_{DP}(G_i,m)$, $L_i(u_i) = \{(u_i,j) : j \in [m]\}$, and $L_i(v_i) = \{(v_i,j) : j \in [m]\}$ for each $i \in [2]$. Assume $(v_i,j)(u_i,j) \in E(H_i)$ for each $i \in [2]$ and $j \in [m]$. Let $a_{j_1,j_2} = N(\{(v_1,j_1),(u_1,j_2)\},\HH_1)$, $b_{j_1,j_2} = N(\{(v_2,j_1),(u_2,j_2)\},\HH_2)$, $X = [m]$, and $D_{\sigma} = \sum_{j_2=1}^{m} \sum_{j_1=1}^{m} a_{j_1,j_2}b_{\sigma(j_1),\sigma(j_2)}$ where $\sigma \in S_X$ and $S_X$ is the symmetric group on $X$. By Lemma~\ref{lem: upperGen2}, we know that $P_{DP}(G,m) \leq D_{\sigma}$ for each $\sigma \in S_X$.
Notice $P_{DP}(G_1,m) = \sum_{j_2=1}^{m} \sum_{j_1=1}^{m} a_{j_1,j_2}$, $P_{DP}(G_2,m) = \sum_{j_2=1}^{m} \sum_{j_1=1}^{m} b_{j_1,j_2}$, and $a_{j,j} = b_{j,j} = 0$ for each $j \in X$. Suppose $q,r,j_1,j_2 \in X$ and $q \not= r$. Notice that $b_{q,r}$ occurs $(m-2)!$ times in $\sum_{\sigma \in S_X} b_{\sigma(j_1),\sigma(j_2)}$ which means $\sum_{\sigma \in S_X} b_{\sigma(j_1),\sigma(j_2)} = (m-2)!P_{DP}(G_2,m)$.
Thus,
\begin{align*}
    \sum_{\sigma \in S_X} D_{\sigma} = \sum_{\sigma \in S_X} \sum_{j_2=1}^{m} \sum_{j_1=1}^{m} a_{j_1,j_2}b_{\sigma(j_1),\sigma(j_2)} &= \sum_{j_2=1}^{m} \sum_{j_1=1}^{m} a_{j_1,j_2} \sum_{\sigma \in S_X} b_{\sigma(j_1),\sigma(j_2)}\\
    &= P_{DP}(G_1,m)(m-2)!P_{DP}(G_2,m).
\end{align*}
So, there is a $\sigma' \in S_X$ such that $$P_{DP}(G,m) \leq D_{\sigma'} \leq \frac{\sum_{\sigma \in S_X} D_{\sigma}}{m!} = \frac{P_{DP}(G_1,m)P_{DP}(G_2,m)}{m(m-1)}.$$

Now suppose $n \geq 3$ and the result holds for all natural numbers greater than $1$ and less than $n$. Let $G'$ be the graph obtained by identifying $u_1,\ldots,u_{n-1}$ as the same vertex $u'$ and $v_1,\ldots,v_{n-1}$ as the same vertex $v'$. By the inductive hypothesis, for each $m \in \N$, $$P_{DP}(G',m) \leq \frac{\prod_{i=1}^{n-1} P_{DP}(G_i,m)}{(m(m-1))^{n-2}}.$$ Notice $G$ is the graph obtained by identifying $u'$ and $u_n$ as the same vertex $u$ and $v'$ and $v_n$ as the same vertex $v$. Thus, by the inductive hypothesis, $$P_{DP}(G,m) \leq \frac{P_{DP}(G',m)P_{DP}(G_n,m)}{m(m-1)} \leq \frac{\prod_{i=1}^{n} P_{DP}(G_i,m)}{(m(m-1))^{n-1}}.$$
\end{proof}

We now wish to present an application of the ideas presented in this section.  Specifically, we will use the ideas developed in this section to determine the DP color function of all chorded cycles.  First, we mention a result from~\cite{BH21}.

\begin{pro} [\cite{BH21}] \label{cor: lowerGenPretty} 
Suppose that $G_1,\ldots,G_n$ are vertex disjoint graphs where $n \geq 2$ and $G \in \bigoplus_{i=1}^{n} (G_i,p)$ where for each $i \in [n]$, $\{u_{i,1},\ldots,u_{i,p}\}$ is a clique in $G_i$ and $G$ is obtained by identifying $u_{1,q},\ldots,u_{n,q}$ as the same vertex $u_q$ for each $q \in [p]$. Suppose that for each $i \in [n]$, given any $m$-fold cover $\DD_i = (L_i,D_i)$ of $G_i$, $N(A,\DD_i) \geq P_{DP}(G_i,m) / \prod_{i=0}^{p-1} (m-i)$ whenever $A \subseteq \bigcup_{q=1}^{p} L(u_{i,q})$, $|A \cap L(u_{i,q})| = 1$ for each $q \in [p]$, and $A$ is an independent set in $D_i$. Then $$P_{DP}(G,m) \geq \frac{\prod_{i=1}^{n} P_{DP}(G_i,m)}{\left(\prod_{i=0}^{n-1} (m-i)\right)^{n-1}}.$$
\end{pro}

Now, we have a corollary that follows immediately from Proposition~\ref{cor: lowerGenPretty} and Theorem~\ref{thm: upperbound2}.

\begin{cor}\label{thm: general2}
Suppose that $G_1,\ldots,G_n$ are vertex disjoint graphs for some $n \geq 2$ with $u_iv_i \in E(G_i)$ for each $i \in [n]$. Suppose that $G$ is the graph obtained by identifying $u_1,\ldots,u_n$ as the same vertex $u$ and $v_1,\ldots,v_n$ as the same vertex $v$. Also suppose that for each $i \in [n]$ and any $m$-fold cover $\HH_i = (L_i,H_i)$ of $G_i$ with $m \geq 2$, $N(\{p_1,p_2\},\HH_i) \geq P_{DP}(G_i,m)/(m(m-1))$ for each $p_1 \in L_i(u_i)$ and $p_2 \in L_i(v_i)$ where $p_1p_2 \notin E(H_i)$. Then $$P_{DP}(G,m) = \frac{\prod_{i=1}^{n} P_{DP}(G_i,m)}{(m(m-1))^{n-1}}.$$
\end{cor}

We now show that the hypothesis of Corollary~\ref{thm: general2} is satisfied when $G_i$ is an odd cycle.  This will ultimately allow us to determine the DP color function of any chorded cycle that can be constructed by edge-gluing two odd cycles.

\begin{lem}\label{lem: lowerOddCyc2}
Suppose that $G = C_{2k+1}$ where $k \in \N$. Also suppose that $\HH = (L,H)$ is an arbitrary $m$-fold cover of $G$ where $m \geq 2$. For each $uv \in E(G)$, whenever $p_1 \in L(u)$, $p_2 \in L(v)$, and $p_1p_2 \notin E(H)$, $N(\{p_1,p_2\},\HH) \geq P_{DP}(G,m)/(m(m-1))$.
\end{lem}
\begin{proof}
Suppose that the vertices of $G$ in cyclic order are $s_1,\ldots,s_{2k+1}$. We can assume that $\HH$ is full.  Let $G' = G - \{s_1s_2\}$. Since $G'$ is a tree, the subcover of $\HH$ corresponding to $G'$, $\HH' = (L,H')$, has a canonical labeling by Proposition~\ref{pro: tree}.  Suppose the vertices of $H$ are named according to this canonical labeling.  Let $G'' = G \cdot s_1s_2$ where the vertex obtained from contracting $s_1s_2$ is $s'$.~\footnote{When $k=1$, we ignore the multiple edge so that $G'' = P_2$. Throughout this paper, we assume that $C_2 = P_2$. Notice that $P(C_2,m) = (m-1)^2 + (m-1) = m(m-1) = P(P_2,m)$.} Suppose $p_1 = (s_1,j_1)$, $p_2 = (s_2,j_2)$, and $p_1p_2 \notin E(H)$. Notice that $N(\{p_1,p_2\},\HH) = N(\{p_1,p_2\},\HH')$. We will show that $N(\{p_1,p_2\},\HH') \geq P_{DP}(G,m)/(m(m-1))$ for each of the following cases: (1) $j_1 \not= j_2$ and (2) $j_1 = j_2$.

In case (1), assume that $j_1 \not= j_2$. Notice $N(\{p_1,p_2\},\HH')$ is the number of proper $m$-colorings of $G$ that color $s_1$ with $j_1$ and $s_2$ with $j_2$. Thus, $$N(\{p_1,p_2\},\HH') = \frac{P(C_{2k+1},m)}{m(m-1)} = \frac{P_{DP}(C_{2k+1},m)}{m(m-1)}.$$
In case (2), assume that $j_1 = j_2 = j$. Notice $N(\{p_1,p_2\},\HH')$ is the number of proper $m$-colorings of $G''$ that color $s'$ with $j$. Thus, $$N(\{p_1,p_2\},\HH') = \frac{P(C_{2k},m)}{m} = \frac{(m-1)^{2k+1} + (m-1)^2}{m(m-1)} > \frac{P_{DP}(C_{2k+1},m)}{m(m-1)}.$$
\end{proof}

The next result easily follows from Corollary~\ref{thm: general2} and Lemma~\ref{lem: lowerOddCyc2}.
\begin{cor}\label{pro: oddCycleEdge}
Suppose that $G_1,\ldots,G_n$ are vertex disjoint odd cycles with $u_iv_i \in E(G_i)$ for each $i \in [n]$. Suppose that $G$ is the graph obtained by identifying $u_1,\ldots,u_n$ as the same vertex $u$ and by identifying $v_1,\ldots,v_n$ as the same vertex $v$. Then for each $m \geq 2$, $$P_{DP}(G,m) = \frac{\prod_{i=1}^{n} P_{DP}(G_i,m)}{(m(m-1))^{n-1}}.$$
\end{cor}

We are now ready to determine the DP color function of all chorded cycles.

\begin{pro} \label{pro: cyclechord}
Suppose that $G_i = C_{n_i}$ where $i \in [2]$ and $n_i \geq 3$. Suppose $u_iv_i \in E(G_i)$ for each $i \in [2]$. Let $G$ be the graph obtained by identifying $u_1$ and $u_2$ as the same vertex $u$ and by identifying $v_1$ and $v_2$ as the same vertex $v$. Then for each $m \geq 3$, if $n_1$ and $n_2$ are even, then $$P_{DP}(G,m) = \frac{1}{m}\left((m-1)^{n_1 + n_2 - 1} - (m-1)^{n_1 - 1} - (m-1)^{n_2 - 1} - m - 1\right).$$ Otherwise, $$P_{DP}(G,m) = \frac{P_{DP}(G_1,m)P_{DP}(G_2,m)}{m(m-1)}.$$
\end{pro}
\begin{proof}
By Corollary~\ref{pro: oddCycleEdge}, $P_{DP}(G,m) = P_{DP}(G_1,m)P_{DP}(G_2,m)/(m(m-1))$ when $n_1$ and $n_2$ are odd.

Let $\HH = (L,H)$ be an $m$-fold cover of $G$ such that $P_{DP}(G,\HH) = P_{DP}(G,m)$. We can assume that $\HH$ is full. For each $i \in [2]$, let $\HH_i = (L_i,H_i)$ be the separated cover of $G_i$ obtained from $\HH$ where $u_{i,1} = u_i$ and $u_{i,2} = v_i$. Let $f$ be the identity permutation of $[m]$. Then $\HH$ is the $f$-amalgamated $m$-fold cover of $G$ obtained from $\HH_1$ and $\HH_2$ (By Definition~\ref{defn: edge amalgamation}, recall that we assume the vertices of $L_i(v_i)$ are renamed so that $(u_i,j)(v_i,j) \in E(H_i)$ for each $i \in [2]$ and $j \in [m]$). By Lemma~\ref{lem: upperGen2}, \\ $P_{DP}(G,\HH) = \sum_{j_2=1}^{m} \sum_{j_1=1}^{m} \left(N(\{(u_1,j_1),(v_1,j_2)\},\HH_1)N(\{(u_2,j_1),(v_2,j_2)\},\HH_2)\right)$.

For each $i \in [2]$, let $\HH_i' = (L_i,H_i')$ be the subcover of $\HH_i$ corresponding to $G_i'$ where $G_i' = G_i - \{u_iv_i\}$. Since $G_i'$ is a tree, $\HH_i'$ has a canonical labeling by Proposition~\ref{pro: tree}.  Suppose the vertices of $H_i'$ are named according to a canonical labeling. Let $r_i$ and $r_i'$ be the permutations of $[m]$ so that $(u_i,j) \in V(H_i)$ is called $(u_i,r_i(j))$ in $H_i'$ and $(v_i,j) \in V(H_i)$ is called $(v_i,r_i'(j))$ in $H_i'$. Notice that when $j_1 \not= j_2$, $N(\{(u_i,j_1),(v_i,j_2)\},\HH_i) = N(\{(u_i,r_i(j_1)),(v_i,r_i'(j_2))\},\HH_i')$ for $i \in [2]$. Also, if $j_1=j_2$, $N(\{(u_i,j_1),(v_i,j_2)\},\HH_i) = 0$ for $i \in [2]$.  Let $G_i'' = G_i \cdot u_iv_i$, and suppose that $w_i$ is the vertex obtained from contracting $u_iv_i$.  Assume that $a_1$ and $a_2$ are fixed elements of $[m]$ such that $a_1 \neq a_2$.

Consider the case where $n_i = 2k + 1$ for some $k \in \N$.  By Lemma~\ref{lem: lowerOddCyc2}, \\ $N(\{(u_i,r_i(a_1)),(v_i,r_i'(a_2))\},\HH_i') \geq P_{DP}(G_i,m)/(m(m-1)) = ((m-1)^{2k} - 1)/m$.

Now, consider the case where $n_i = 2l + 2$ for some $l \in \N$. Then when $r_i(a_1) \not= r_i'(a_2)$, since $H_i'$ has a canonical labeling, $N(\{(u_i,r_i(a_1)),(v_i,r_i'(a_2))\},\HH_i')$ is the number of proper $m$-colorings of $G_i$ that color $u_i$ with $r_i(a_1)$ and $v_i$ with $r_i'(a_2)$. Thus, $$N(\{(u_i,r_i(a_1)),(v_i,r_i'(a_2))\},\HH_i') = \frac{P(C_{2l+2},m)}{m(m-1)} = \frac{(m-1)^{2l+1} + 1}{m}.$$ When $r_i(a_1) = r_i'(a_2) = j$, since $H_i'$ has a canonical labeling, $N(\{(u_i,r_i(a_1)),(v_i,r_i'(a_2))\},\HH_i')$ is the number of proper $m$-colorings of $G_i''$ that color $w_i$ with $j$. Thus, $$N(\{(u_i,r_i(a_1)),(v_i,r_i'(a_2))\},\HH_i') = \frac{P(C_{2l+1},m)}{m} = \frac{(m-1)^{2l+1} - (m-1)}{m}.$$ Notice there are at most $m$ pairs $(j_1,j_2)$ in the set $[m]^2 - \{(i,i) : i \in [m]\}$ such that $N(\{(u_i,r_i(j_1)),(v_i,r_i'(j_2))\},\HH_i') = ((m-1)^{2l+1} - (m-1))/m$ since there can be at most $m$ pairs $(j_1,j_2) \in [m]^2$ such that $r_i(j_1) = r_i'(j_2)$. This implies that there are at least $m(m-2)$ pairs $(j_1,j_2) \in ([m]^2 - \{(i,i) : i \in [m]\})$ such that $N(\{(u_i,r_i(j_1)),(v_i,r_i'(j_2))\},\HH_i') = ((m-1)^{2l+1} + 1)/m$.

We now are ready to prove the result when exactly one of the numbers $n_1, n_2$ is odd.  Suppose without loss of generality $n_1 = 2k + 1$ and $n_2 = 2l + 2$ for some $k,l \in \N$. We calculate
\begin{align*}
    &P_{DP}(G,\HH) \\
		&= \sum_{(j_1,j_2) \in [m]^2, \; j_1 \neq j_2} N(\{(u_1,j_1),(v_1,j_2)\},\HH_1) N(\{(u_2,j_1),(v_2,j_2)\},\HH_2)\\
    &=\sum_{(j_1,j_2) \in [m]^2, \; j_1 \neq j_2}  N(\{(u_1,r_1(j_1)),(v_1,r_1'(j_2))\},\HH_1') N(\{(u_2,r_2(j_1)),(v_2,r_2'(j_2))\},\HH_2')\\
    &\geq \frac{(m-1)^{2k} - 1}{m} \sum_{(j_1,j_2) \in [m]^2, \; j_1 \neq j_2}  N(\{(u_2,r_2(j_1)),(v_2,r_2'(j_2))\},\HH_2')\\
    &\geq \frac{(m-1)^{2k} - 1}{m} \left(m\frac{(m-1)^{2l+1} - (m-1)}{m} + m(m-2)\frac{(m-1)^{2l+1} + 1}{m}\right)\\
    &= \frac{(m-1)^{2k+1} - (m-1)}{m(m-1)} \left((m-1)^{2l+2} - 1\right)
    = \frac{P_{DP}(G_1,m)P_{DP}(G_2,m)}{m(m-1)}.
\end{align*}
This calculation along with Theorem~\ref{thm: upperbound2} implies that \\ $P_{DP}(G,m) = P_{DP}(G_1,m)P_{DP}(G_2,m)/(m(m-1))$.  

Finally, we turn our attention to the case where both $n_1$ and $n_2$ are even.  Suppose $n_1 = 2k + 2$ and $n_2 = 2l + 2$ for some $k,l \in \N$. Using the facts above and the rearrangement inequality (see~\cite{R52}), we obtain:
\begin{align*}
    &P_{DP}(G,\HH) \\ 
		&= \sum_{(j_1,j_2) \in [m]^2, \; j_1 \neq j_2} N(\{(u_1,j_1),(v_1,j_2)\},\HH_1) N(\{(u_2,j_1),(v_2,j_2)\},\HH_2)\\
    &= \sum_{(j_1,j_2) \in [m]^2, \; j_1 \neq j_2} N(\{(u_1,r_1(j_1)),(v_1,r_1'(j_2))\},\HH_1') N(\{(u_2,r_2(j_1)),(v_2,r_2'(j_2))\},\HH_2') \\
    &\geq m\left(\frac{(m-1)^{2l+1} - (m-1)}{m}\right)\left(\frac{(m-1)^{2k+1} + 1}{m}\right) \\
		&+ m\left(\frac{(m-1)^{2l+1} + 1}{m}\right)\left(\frac{(m-1)^{2k+1} - (m-1)}{m}\right) \\
		&+ m(m-3)\left(\frac{(m-1)^{2l+1} + 1}{m}\right)\left(\frac{(m-1)^{2k+1} + 1}{m}\right)\\
    &= \frac{1}{m}\left((m-1)^{2l+2k+3} - (m-1)^{2l+1} - (m-1)^{2k+1} - m - 1\right)\\
    &= \frac{1}{m}\left((m-1)^{n_1 + n_2 - 1} - (m-1)^{n_2 - 1} - (m-1)^{n_1 - 1} - m - 1\right).
\end{align*}
Thus, $P_{DP}(G,m) \geq \left((m-1)^{n_1 + n_2 - 1} - (m-1)^{n_2 - 1} - (m-1)^{n_1 - 1} - m - 1\right)/m$.  

To finish the proof we will construct an $m$-fold cover $\HH' = (L',H')$ of $G$ such that \\ $P_{DP}(G,\HH') = \left((m-1)^{n_1 + n_2 - 1} - (m-1)^{n_2 - 1} - (m-1)^{n_1 - 1} - m - 1\right)/m$. In order to build $\HH'$, we will first construct an $m$-fold cover $\HH'_i = (L'_i,H'_i)$ of $G_i$ for each $i \in [2]$. For each $i \in [2]$, let $L'_i(x) = \{(x,j) : j \in [m]\}$ whenever $x \in V(G_i)$. Construct edges in $H'_i$ so that $L'_i(x)$ is a clique in $H'_i$ for each $x \in V(G_i)$.  Then, for each $xy \in E(G_i)-\{u_iv_i\}$ construct the edge $(x,j)(y,j)$ in $H'_i$ for each $j \in [m]$.  Finally, construct edges in $H'_1$ so that $\{(u_1,j)(v_1,j+1) : j \in [m-1]\} \cup \{(u_1,m)(v_1,1)\} \subset E(H'_1)$, and construct edges in $H'_2$ so that $\{(u_2,j)(v_2,j+2) : j \in [m-2]\} \cup \{(u_2,m-1)(v_2,1),(u_2,m)(v_2,2)\} \subset E(H'_2)$. 

Now, for each $i \in [2]$, let $\HH''_i = (L'_i,H''_i)$ be the $m$-fold cover of $G_i$ where $H''_i$ is the same graph as $H'_i$ except the vertices of $L'_i(v_i)$ are renamed so that $(u_i,j)(v_i,j) \in E(H_i)$ for each $j \in [m]$.  Specifically, for any $j \in [m]$, the vertex $(v_1,j) \in V(H''_1)$ is called $(v_1, (\text{$j$ mod $m$})+1)$ in $H'_1$, and the vertex $(v_2,j) \in V(H''_2)$ is called $(v_2, (\text{$j+1$ mod $m$})+1)$ in $H'_2$.  Let $f$ be the identity permutation of $[m]$ and $\HH$ be the $f$-amalgamated $m$-fold cover of $G$ obtained from $\HH''_1$ and $\HH''_2$. By Lemma~\ref{lem: upperGen2}, 
\begin{align*}
P_{DP}(G,\HH') &= \sum_{j_2=1}^{m} \sum_{j_1=1}^{m} N(\{(u_1,j_1),(v_1,j_2)\},\HH''_1) N(\{(u_2,j_1),(v_2,j_2)\},\HH''_2) \\
&= \sum_{(j_1,j_2) \in [m]^2, \; j_1 \neq j_2} N(\{(u_1,j_1),(v_1,j_2)\},\HH''_1) N(\{(u_2,j_1),(v_2,j_2)\},\HH''_2). 
\end{align*}
Furthermore, we know $N(\{(u_1,j_1),(v_1,j_2)\},\HH''_1) =  N(\{(u_1,j_1),(v_1, (\text{$j_2$ mod $m$})+1)\},\HH'_1)$ and $N(\{(u_2,j_1),(v_2,j_2)\},\HH''_2) = N(\{(u_2,j_1),(v_2,(\text{$j_2+1$ mod $m$})+1)\},\HH'_2)$.  So, \\ $N(\{(u_1,j_1),(v_1,j_2)\},\HH''_1) = ((m-1)^{2k+1} - (m-1))/m$ when $j_1 = (\text{$j_2$ mod $m$})+1$, and $N(\{(u_1,j_1),(v_1,j_2)\},\HH''_1) = ((m-1)^{2k+1} + 1)/m$ for all other $(j_1,j_2) \in ([m]^2 - \{(i,i) : i \in [m]\})$.  Similarly, $N(\{(u_2,j_1),(v_1,j_2)\},\HH''_2) = ((m-1)^{2l+1} - (m-1))/m$ when $j_1 = (\text{$j_2+1$ mod $m$})+1$, and $N(\{(u_2,j_1),(v_2,j_2)\},\HH''_2) = ((m-1)^{2l+1} + 1)/m$ for all other $(j_1,j_2) \in ([m]^2 - \{(i,i) : i \in [m]\})$.  An easy calculation then yields:
$$P_{DP}(G,\HH') =\frac{1}{m}\left((m-1)^{n_1 + n_2 - 1} - (m-1)^{n_2 - 1} - (m-1)^{n_1 - 1} - m - 1\right)$$
which completes the proof.
\end{proof}

\section{A Counterexample for Triangle-Gluings} \label{triangle}

In this section we will show that there exists a graph $G$ such that $G^*$, a triangle gluing of $G$ and $K_4$, satisfies $P_{DP}(G^*,m) > P_{DP}(G,m) P_{DP}(K_4,m)/(m(m-1)(m-2))$ for $m=4$. This gives us a counterexample to an affirmative answer for Questions~\ref{ques: classify} and~\ref{ques: easier} with $p=3$ and $d_G(v)=3$ respectively.

Let $G_0 = K_1  \vee \Theta(2,2,2)$. Let $w$ be the universal vertex in $G_0$. In the copy of $\Theta(2,2,2)$, let $v_1$, $v_2$ be the degree 3 vertices and $u_1, u_2, u_3$ be the degree 2 vertices. In the following, we will denote $G_0[\{v_1,v_2,u_1,u_2,u_3\}]$ by $G'$.

\begin{lem} \label{lem: PDPcounterex}
	$P_{DP}(G_0,4) \le 104 < P(G_0,4) =120$.
\end{lem}
\begin{proof}
	Since $\Theta(2,2,2)$ is isomorphic to $K_{2,3}$, and $P(K_{2,3},3) = 30$, it follows that $P(G_0,4) = 120$. To complete the proof, we now argue that $P_{DP}(G_0,4) \le 104$.
	
	We construct a $4$-fold cover $\mathcal{H} = (L,H)$ of $G_0$. Let $L(u) = \{ (u,j) : j \in [4] \}$ for $u \in V(G_0)$. The cross-edges are defined for each $j \in [4]$ as: $(u,j)(v,j) \in E(H)$ whenever $uv \in (E(G_0) - \{v_1u_2, v_1u_3\})$; $(v_1,j)(u_2,(j \mod 4)+1) \in E(H)$;  $(v_1,j)(u_3,(j + 1 \mod 4)+1) \in E(H)$.
	
We claim that $N((w,j),\mathcal{H}) = 26$ for each $j \in [4]$. This will imply that $P_{DP}(G_0,\mathcal{H}) = 104$, since $w$ is the universal vertex in $G_0$.
	
	
We will count the number of possibilities for picking vertices in $H$ to form an $\mathcal{H}$-coloring $I$ of $G_0$. Without loss of generality, suppose $(w,1) \in I$.
	Then $I$ can not contain any vertices in $\{(u_i,1) : i \in [3]\} \cup \{(v_j,1) : j \in [2] \}$. Let $H'$ be the subgraph of $H$ that remains after removing $L(w)$ and the five vertices in this set. Let $L'(u) = \{ (u,j) : j \in \{2,3,4\}\}$ for $u \in \{v_1,v_2,u_1,u_2,u_3\}$. Then,  $\mathcal{H}'=(L',H')$ is a 3-fold cover of $G'$ (recall $G'$ is a copy of $\Theta(2,2,2)$, or equivalently $K_{2,3}$). Clearly $N((w,1),\mathcal{H})$ equals $P_{DP}(G',\mathcal{H}')$, the number of $\mathcal{H}'$-colorings of $G'$.
	
	Let $\mathcal{H}''$ be the subcover of $\mathcal{H}'$ corresponding to $G'' = G' - \{v_1u_2, v_1u_3\}$.  It is immediately clear that $\HH''$ is a full 3-fold cover of $G''$. Since $G''$ is a tree, $\mathcal{H}''$ has a canonical labeling (by Proposition~\ref{pro: tree}) and there are 48  $\mathcal{H}''$-colorings of $G''$.  To find $P_{DP}(G',\mathcal{H}')$, we need to determine the number of these $\mathcal{H}''$-colorings of $G''$ that are not also $\mathcal{H}'$-colorings of $G'$ and subtract this number from 48.
	
	Note $\{(v_1,2)(u_2,3), (v_1,3)(u_2,4) \} = E_{H'}(L'(v_1), L'(u_2))$.  Let $A$ be the set of all $\mathcal{H}''$-colorings containing $(v_1,2)$ and $(u_2,3)$, and $B$ be the set of all $\mathcal{H}''$-colorings containing $(v_1,3)$ and $(u_2,4)$. Then $|A| = |B| = 6$.
	
	Similarly, note $\{(v_1,2)(u_3,4), (v_1,4)(u_3,2) \} = E_{H'}(L'(v_1), L'(u_3))$. Let $C$ be the set of all $\mathcal{H}''$-colorings containing $(v_1,2)$ and $(u_3,4)$, and $D$ be the set of all $\mathcal{H}''$-colorings containing $(v_1,4)$ and $(u_3,2)$. Then $|C| = |D| = 6$. 
	
Notice $A,B,C,D$ are pairwise disjoint except $|A \cap C| =2$. Hence $|A \cup B \cup C \cup D| = 22$. Therefore, there are $48 -22 = 26$ $\mathcal{H}'$-colorings of $G'$ and $N((w,1),\mathcal{H}) = 26$, as claimed.
\end{proof}

Next, we will prove that any ``precoloring'' of vertices of a $K_3$ in $G_0$ can be extended to an $\mathcal{H}$-coloring of $G_0$ for any 4-fold cover $\mathcal{H}$ of $G_0$. In the following, if $x,y,z$ are the vertices of a clique of order 3 in a graph $G$, we will say that \emph{$xyz$ is a triangle in $G$}.

\begin{lem} \label{lem: precolorcounterex}
Let $x_1x_2x_3$ be a triangle in $G_0$. Let $\mathcal{H} = (L, H)$ be any 4-fold cover of $G_0$. Then, $N(\{p_1,p_2,p_3\}, \mathcal{H}) \ge 1$ whenever $\{p_1,p_2,p_3\}$, with $p_i \in L(x_i)$, is an independent set in $\mathcal{H}$.	
\end{lem}
\begin{proof}
 Note that $w$, the universal vertex in $G_0$, must be one of  $x_1,x_2,x_3$. When we remove $x_1,x_2,x_3$ from $G_0$, we are left with a copy of $P_3$. Let $y_1$ be the vertex of highest degree in $G_0 - \{x_1,x_2, x_3\}$, and $y_2$, $y_3$ be the remaining two degree 1 vertices. Starting with the given vertices $\{p_1,p_2,p_3\}$, we can greedily complete an $\mathcal{H}$-coloring of $G_0$ picking either of two available vertices in $L(y_1)$, followed by the vertices available in $L(y_2)$ and $L(y_3)$. 
\end{proof}

The final ingredient in the construction is the following operation (sometimes called a {\sl tessellation}) on a graph $G$ containing a triangle $abc$: $T(G,abc)$ is the graph obtained from $G$ by adding a new vertex $d$ along with edges that make $d$ a common neighbor to the vertices $a,b,c$. Note that $T(G,abc)$ creates a graph that is equivalent to a triangle-gluing of $G$ with a $K_4$.

$G_0$ has a total of 6 triangles (each of the form $wv_iu_j$, $i \in\{1,2\}, j\in \{1,2,3\})$. Let us name the triangle $wv_1u_1$ as $t_1$ and arbitrarily name the remaining five triangles as $t_i$, $i \in \{2,\ldots,6\}$. Let $G_k = T(G_{k-1}, t_k)$ for each $k \in [6]$. We will use the name $d_i$ for the new vertex introduced in $G_i$ for each $i \in [6]$.

\begin{lem} \label{lem: Finalcounterex}
	If $G_1, \ldots, G_5$ are not counterexamples for an affirmative answer to Question~\ref{ques: classify}, then $G_6$ is such a counterexample.
\end{lem}
\begin{proof}
Let $G^* = G_6$.
Note that $P_{DP}(K_4,4)= 24$. So, the hypotheses of the Lemma imply that for each $i \in [5]$, $ P_{DP}(G_i,4) \le P_{DP}(G_{i-1},4)  P_{DP}(K_4,4)/ 24$; that is, $P_{DP}(G_0,4) \ge P_{DP}(G_1,4) \ge P_{DP}(G_2,4) \ge P_{DP}(G_3,4) \ge P_{DP}(G_4,4) \ge P_{DP}(G_5,4)$.  To show $G^*$ is a counterexample, we will show that $P_{DP}(G^*,4) >  P_{DP}(G_0,4)$.

For the sake of contradiction, assume $P_{DP}(G^*,4) \le  P_{DP}(G_0,4)$.  We know $P_{DP}(G_0,4) <  P(G_0,4)$ by Lemma~\ref{lem: PDPcounterex}.  Let $\mathcal{H}^*= (L^*, H^*)$ be a 4-fold cover of $G^*$ that gives the minimum number of DP-colorings of $G^*$,  $P_{DP}(G^*,4) =  P_{DP}(G^*,\mathcal{H}^*)$. Let $\mathcal{H}_0 = (L_0,H_0)$ be the subcover of $\mathcal{H}^*$ induced by $V(G_0)$.


Suppose $\mathcal{H}_0$ has a canonical labeling.  Then we have $P(G_0,4) = 120$ $\mathcal{H}_0$-colorings of $G_0$. Since each of the new tessellation vertices $d_1, \ldots, d_6$ in $G^*$ is of degree 3, each of these $\mathcal{H}_0$-colorings can be extended to an $\mathcal{H}^*$ of $G^*$. This contradicts $P_{DP}(G^*,4) < P(G_0,4)$. Thus, $\mathcal{H}_0$ has no canonical labeling.

We will now rename the vertices of $H_0$ and $H^*$ as $L^*(u) = L_0(u) = \{ (u,j) : j \in [m] \}$ for all $u \in V(G_0) - \{w\}$ in such way that the cross-edges of $H_0$ incident to vertices in $L_0(w)$ are of the form $(w,j)(v,j)$ for each $v\ne w$ in $V(G_0)$.  Since $\mathcal{H}_0$ has no canonical labeling, at least one of the matchings $E_{H_0}(L(u),L(v))$, where $u,v \in V(G_0)-\{w\}$, is \emph{twisted}; that is, there exists an $l \in [4]$ such that $(u,l)(v,l)\not\in E(H_0)$.  Without loss of generality, we can assume that the matching $E_{H_0}(L(u_1),L(v_1))$ is twisted.  Importantly, this implies that if $\HH_1 = (L_1, H_1)$ is the subcover of $\HH^*$ induced by $\{w, u_1, v_1 \}$, then $\HH_1$ does not have a canonical labeling (see the argument for Lemma~14 in~\cite{KM21}). Let $d= d_1$.  Recall $d$ is the tessellation vertex introduced in the triangle $wu_1v_1$. We can now rename the vertices of $L^*(d)$, $L^*(u_1)$, and $L^*(v_1)$ such that the cross-edges of $H^*$ incident to the vertices of $L^*(d)$ are of the form $(d,j)(v,j)$ for each $v \in \{w,u_1,v_1\}$.  Also, rename the vertices of $L_0(u_1)$ and $L_0(v_1)$ in the same way the vertices of $L^*(u_1)$ and $L^*(v_1)$ were renamed.   

Using the naming schemes described above for $\HH_0$ and $\HH^*$, let $\mathcal{C}_0$ be the set of all $\mathcal{H}_0$-colorings of $G_0$, and let $\mathcal{C}^*$ be the set of all $\mathcal{H}^*$-colorings of $G^*$.  Since each of the new tessellation vertices, $d_i$, has degree 3, each $\mathcal{H}_0$-coloring of $G_0$ can be extended to a $\mathcal{H}^*$-coloring of $G^*$. Hence the following function is well-defined.  Let $f : \mathcal{C}_0 \rightarrow \mathcal{P}(\mathcal{C}^*)$ be given by $f_0(I) = \{C \in \mathcal{C}^* \;:\; I \subseteq C\}$.

By definition, $f(I) \cap f(J) = \emptyset$ for $I \neq J$. And, as noted above, $|f(I)| \ge 1$ for all $I \in \mathcal{C}_0$.  This means $P_{DP}(G_0,4) \le P_{DP}(G_0,\mathcal{H}_0) \le \sum_{I\in \mathcal{C}_0} |f(I)| \le |\mathcal{C}^*|=P_{DP}(G^*,\mathcal{H}^*)=P_{DP}(G^*,4)$. Next, we will argue that, in fact $P_{DP}(G_0,\mathcal{H}_0) < \sum_{I\in \mathcal{C}_0} |f(I)|$, which will lead to a contradiction to our assumption that  $P_{DP}(G_0,4) \ge P_{DP}(G^*,4)$.

Let $\mathcal{H}_2 = (L_2,H_2)$ be the subcover of $\mathcal{H}^*$ induced by $\{w,v_1,u_1\}$ (note that $H_1$ and $H_2$ are isomorphic but may have different vertex names) and $\mathcal{C}_2$ be the set of all $\mathcal{H}_2$-colorings of $G_0[\{w,v_1,u_1\}]$. Since $\HH_2$ has no canonical labeling, the argument for Lemma~23 in~\cite{KM19} implies that $P_{DP}(G_0[\{w,v_1,u_1\}],\mathcal{H}_2) > P(C_3,4)$ which means $P_{DP}(G_0[\{w,v_1,u_1\}],\mathcal{H}_2) \ge 25$. So, there exists $A=\{(u_1,i),(v_1,j),(w,l)\} \in \mathcal{C}_2$ such that $|\{i,j,l\}| \le 2$. By Lemma~\ref{lem: precolorcounterex}, there exists a $D \in \mathcal{C}_0$ such that $A \subseteq D$. Since $|L^*(d)-\{(d,i),(d,j),(d,l)\}|\ge 2$, $|f(D)| \ge 2$. This implies  $P_{DP}(G_0,\mathcal{H}_0) < \sum_{I\in \mathcal{C}_0} |f(I)|$, as claimed. 
\end{proof}

\section{General Gluings} \label{general}

The issue with trying to generalize the proof of Theorem~\ref{thm: upperbound2} to $K_p$-gluings with $p \geq 3$ is that Definition~\ref{defn: edge amalgamation} does not easily generalize to $K_3$-gluings. In particular, the part of Definition~\ref{defn: edge amalgamation} where we rename vertices so that $(u_i,j)(v_i,j) \in E(H_i)$ for each $j \in [m]$ may not be possible when a cycle is present in the clique we are gluing. The counterexample in Section~\ref{triangle} shows that this issue is sometimes unavoidable. So, in this Section we prove a weaker result by not considering all full covers; that is, we prove Theorem~\ref{thm: upperboundFixed}.  

Suppose $p \geq 1$ and $G$ is an arbitrary graph such that $K= \{v_1,\ldots,v_p\}$ is a clique in $G$.  Recall that an $m$-fold cover $\HH = (L,H)$ of $G$ is \emph{conducive} to $K$ if $\HH$ is full and the $m$-fold cover $\mathcal{H}_K$ of $G[K]$ admits a canonical labeling. Unless otherwise noted, if $\HH=(L,H)$ is conducive to $\{v_1,\ldots,v_p\}$, we will assume that $L(v_q) = \{(v_q,j) : j \in [m]\}$ for each $q \in [p]$ and that $\{(v_q,j) : q \in [p]\}$ is a clique in $H$ for each $j \in [m]$.

Suppose $G$ is an arbitrary graph, $K$ is a clique in $G$, and $m \in \N$.  The \emph{$K$-canonical DP-Color Function of $G$}, $P_{DP}'(G,K,m)$, is the minimum value of $P_{DP}(G,\HH)$ where the minimum is taken over all $m$-fold covers of $G$ conducive to $K$. Clearly, $P_{DP}'(G,K,m) \geq P_{DP}(G,m)$.  We are now ready to present an important generalization of Definition~\ref{defn: edge amalgamation}.

\begin{defn}\label{defn: complete amalgamation}
For some $n \geq 2$, suppose that $G_1,\ldots,G_n$ are vertex disjoint graphs such that $\{u_{i,1},\ldots,u_{i,p}\}$ is a clique in $G_i$ for each $i \in [n]$. For each $i \in [n]$, suppose $\HH_i = (L_i,H_i)$ is an $m$-fold cover of $G_i$ conducive to $\{u_{i,1},\ldots,u_{i,p}\}$. Let $f_{k+1}$ be a permutation of $[m]$ for each $k \in [n-1]$, and let $F = (f_2,\ldots,f_n)$. Let $G$ be the graph obtained by identifying $u_{1,q},\ldots,u_{n,q}$ as the same vertex $u_q$ for each $q \in [p]$. \textbf{The $F$-amalgamated $m$-fold cover of $G$ obtained from $\HH_1,\ldots,\HH_n$} is an $m$-fold cover $\HH = (L,H)$ of $G$ defined as follows. In the special case where $n = 2$, we may also say $f_2$-amalgamated $m$-fold cover of $G$ obtained from $\HH_1$ and $\HH_2$.

For each $i \in [n]$, assume $L_i(x) = \{(x,j): j \in [m]\}$ for each $x \in V(G_i)$. For each $i \in [n]$, let $X_i$ be the set of edges in $H_i$ that are incident to at least one element in $\bigcup_{q=1}^{p} L_i(u_{i,q})$. Let $L(x) = \{(x,j) : j \in [m]\}$ for each $x \in V(G)$. Construct edges in $H$ so that for each $q \in [p]$, $L(u_q)$ is a clique in $H$. Construct edges in $H$ so that $\bigcup_{i=1}^{n} (E(H_i) - X_i) \subseteq E(H)$. Then for each $j \in [m]$, $k \in [n-1]$, and $q \in [p]$, whenever $(x,r)(u_{1,q},j) \in E(H_1)$ where $(x,r) \notin \bigcup_{q=1}^{p} L_1(u_{1,q})$ or $(x,r)(u_{k+1,q},f_{k+1}(j)) \in E(H_{k+1})$ where $(x,r) \notin \bigcup_{q=1}^{p} L_{k+1}(u_{k+1,q})$, construct the edge $(x,r)(u_q,j)$ in $H$. Finally, construct edges in $H$ such that $\{(u_q,j) : q \in [p]\}$ is a clique in $H$ for each $j \in [m]$.
\end{defn}

Clearly, the $F$-amalgamated $m$-fold cover of $G$ obtained from $\HH_1,\ldots,\HH_n$ is conducive to $\{u_1,\ldots,u_p\}$.  Below is an illustration of Definition~\ref{defn: complete amalgamation} with $n=2$, $m=3$, $G_1= K_4 - e$, and $G_2$ is a $C_3$ with an pendant edge.  Note in the drawing of each cover below only the cross-edges of each cover are shown.

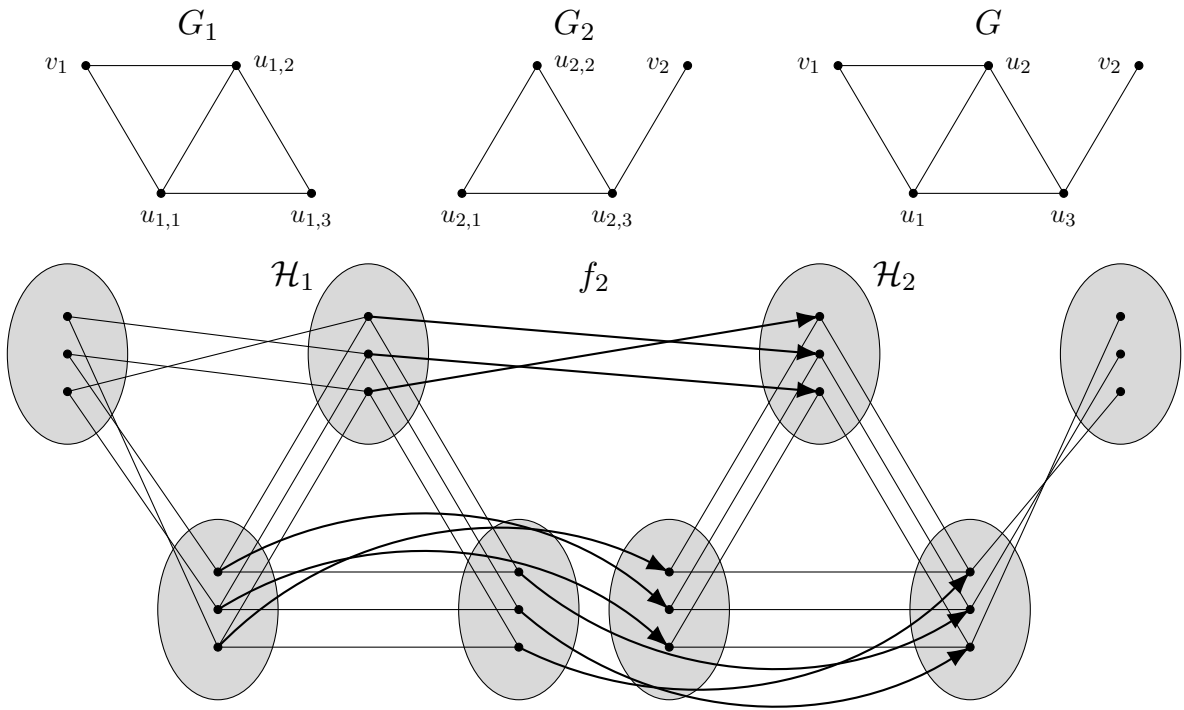
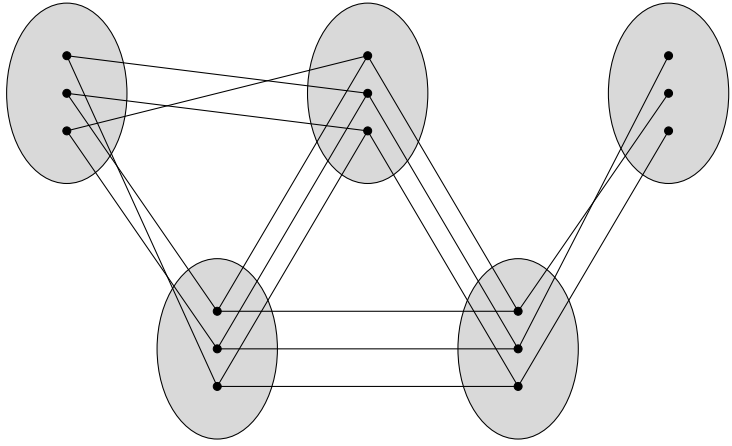
\begin{center}
\begin{tikzpicture}
    \coordinate (v11) at (-3,0);
    \coordinate (v12) at (-2,1.7);
    \coordinate (v01) at (-2.5,2.7);
    \coordinate (v13) at (-1,0);
    \coordinate (v14) at (-4,1.7);
    
    \coordinate (v21) at (1,0);
    \coordinate (v22) at (2,1.7);
    \coordinate (v02) at (2.5,2.7);
    \coordinate (v23) at (3,0);
    \coordinate (v24) at (4,1.7);
    
    \coordinate (v4) at (6,1.7);
    \coordinate (v1) at (7,0);
    \coordinate (v2) at (8,1.7);
    \coordinate (v03) at (8,2.7);
    \coordinate (v3) at (9,0);
    \coordinate (v5) at (10,1.7);
    
    \draw[fill=black] (v11) circle[radius=1.5pt] node[below=3pt,scale=0.9] {$u_{1,1}$};
    \draw[fill=black] (v12) circle[radius=1.5pt] node[right=3pt,scale=0.9] {$u_{1,2}$};
    \draw[fill=black] (v01) circle[radius=0pt] node[below=3pt,scale=1.2] {$G_1$};
    \draw[fill=black] (v13) circle[radius=1.5pt] node[below=3pt,scale=0.9] {$u_{1,3}$};
    \draw[fill=black] (v14) circle[radius=1.5pt] node[left=3pt,scale=0.9] {$v_1$};
    
    \draw[fill=black] (v21) circle[radius=1.5pt] node[below=3pt,scale=0.9] {$u_{2,1}$};
    \draw[fill=black] (v02) circle[radius=0pt] node[below=3pt,scale=1.2] {$G_2$};
    \draw[fill=black] (v22) circle[radius=1.5pt] node[right=3pt,scale=0.9] {$u_{2,2}$};
    \draw[fill=black] (v23) circle[radius=1.5pt] node[below=3pt,scale=0.9] {$u_{2,3}$};
    \draw[fill=black] (v24) circle[radius=1.5pt] node[left=3pt,scale=0.9] {$v_2$};
    
    \draw[fill=black] (v1) circle[radius=1.5pt] node[below=3pt,scale=0.9] {$u_1$};
    \draw[fill=black] (v03) circle[radius=0pt] node[below=3pt,scale=1.2] {$G$};
    \draw[fill=black] (v2) circle[radius=1.5pt] node[right=3pt,scale=0.9] {$u_2$};
    \draw[fill=black] (v3) circle[radius=1.5pt] node[below=3pt,scale=0.9] {$u_3$};
    \draw[fill=black] (v4) circle[radius=1.5pt] node[left=3pt,scale=0.9] {$v_1$};
    \draw[fill=black] (v5) circle[radius=1.5pt] node[left=3pt,scale=0.9] {$v_2$};
    
    \draw (v11) -- (v12);
    \draw (v12) -- (v13);
    \draw (v13) -- (v11);
    \draw (v11) -- (v14);
    \draw (v12) -- (v14);
    
    \draw (v21) -- (v22);
    \draw (v22) -- (v23);
    \draw (v23) -- (v21);
    \draw (v23) -- (v24);
    
    \draw (v1) -- (v2);
    \draw (v2) -- (v3);
    \draw (v3) -- (v1);
    \draw (v1) -- (v4);
    \draw (v2) -- (v4);
    \draw (v3) -- (v5);
\end{tikzpicture}

\begin{tikzpicture}
    \coordinate (v111) at (-5,0);
    \coordinate (v121) at (-3,3.4);
    \coordinate (v131) at (-1,0);
    \coordinate (v141) at (-7,3.4);
    \coordinate (v112) at (-5,-0.5);
    \coordinate (v122) at (-3,2.9);
    \coordinate (v132) at (-1,-0.5);
    \coordinate (v142) at (-7,2.9);
    \coordinate (v113) at (-5,-1);
    \coordinate (v123) at (-3,2.4);
    \coordinate (v133) at (-1,-1);
    \coordinate (v143) at (-7,2.4);
    \coordinate (v01) at (-4,4.4);
    
    \coordinate (v211) at (1,0);
    \coordinate (v221) at (3,3.4);
    \coordinate (v231) at (5,0);
    \coordinate (v241) at (7,3.4);
    \coordinate (v212) at (1,-0.5);
    \coordinate (v222) at (3,2.9);
    \coordinate (v232) at (5,-0.5);
    \coordinate (v242) at (7,2.9);
    \coordinate (v213) at (1,-1);
    \coordinate (v223) at (3,2.4);
    \coordinate (v233) at (5,-1);
    \coordinate (v243) at (7,2.4);
    \coordinate (v02) at (4,4.4);
    
    \coordinate (v0) at (0,4.4);
    
    \draw[fill=gray!30] (-5,-0.5) ellipse (8mm and 12mm);
    \draw[fill=gray!30] (-3,2.9) ellipse (8mm and 12mm);
    \draw[fill=gray!30] (-1,-0.5) ellipse (8mm and 12mm);
    \draw[fill=gray!30] (-7,2.9) ellipse (8mm and 12mm);
    
    \draw[fill=gray!30] (1,-0.5) ellipse (8mm and 12mm);
    \draw[fill=gray!30] (3,2.9) ellipse (8mm and 12mm);
    \draw[fill=gray!30] (5,-0.5) ellipse (8mm and 12mm);
    \draw[fill=gray!30] (7,2.9) ellipse (8mm and 12mm);
    
    \draw[fill=black] (v111) circle[radius=1.5pt]; 
    \draw[fill=black] (v121) circle[radius=1.5pt]; 
    \draw[fill=black] (v131) circle[radius=1.5pt]; 
    \draw[fill=black] (v141) circle[radius=1.5pt]; 
    \draw[fill=black] (v112) circle[radius=1.5pt]; 
    \draw[fill=black] (v122) circle[radius=1.5pt]; 
    \draw[fill=black] (v132) circle[radius=1.5pt]; 
    \draw[fill=black] (v142) circle[radius=1.5pt]; 
    \draw[fill=black] (v113) circle[radius=1.5pt]; 
    \draw[fill=black] (v123) circle[radius=1.5pt]; 
    \draw[fill=black] (v133) circle[radius=1.5pt]; 
    \draw[fill=black] (v143) circle[radius=1.5pt]; 
    \draw[fill=black] (v01) node[below=3pt,scale=1.2] {$\HH_1$};
    
    \draw[fill=black] (v211) circle[radius=1.5pt]; 
    \draw[fill=black] (v221) circle[radius=1.5pt]; 
    \draw[fill=black] (v231) circle[radius=1.5pt]; 
    \draw[fill=black] (v241) circle[radius=1.5pt]; 
    \draw[fill=black] (v212) circle[radius=1.5pt]; 
    \draw[fill=black] (v222) circle[radius=1.5pt]; 
    \draw[fill=black] (v232) circle[radius=1.5pt]; 
    \draw[fill=black] (v242) circle[radius=1.5pt]; 
    \draw[fill=black] (v213) circle[radius=1.5pt]; 
    \draw[fill=black] (v223) circle[radius=1.5pt]; 
    \draw[fill=black] (v233) circle[radius=1.5pt]; 
    \draw[fill=black] (v243) circle[radius=1.5pt]; 
    \draw[fill=black] (v02) node[below=3pt,scale=1.2] {$\HH_2$};
    
    \draw (v111) -- (v121);
    \draw (v121) -- (v131);
    \draw (v131) -- (v111);
    \draw (v111) -- (v142);
    \draw (v121) -- (v143);
    \draw (v112) -- (v122);
    \draw (v122) -- (v132);
    \draw (v132) -- (v112);
    \draw (v112) -- (v143);
    \draw (v122) -- (v141);
    \draw (v113) -- (v123);
    \draw (v123) -- (v133);
    \draw (v133) -- (v113);
    \draw (v113) -- (v141);
    \draw (v123) -- (v142);
    
    \draw (v211) -- (v221);
    \draw (v221) -- (v231);
    \draw (v231) -- (v211);
    \draw (v231) -- (v243);
    \draw (v212) -- (v222);
    \draw (v222) -- (v232);
    \draw (v232) -- (v212);
    \draw (v232) -- (v242);
    \draw (v213) -- (v223);
    \draw (v223) -- (v233);
    \draw (v233) -- (v213);
    \draw (v233) -- (v241);
    
    \draw [-{Latex[length=3mm]}, thick] (v121) -- (v222);
    \draw [-{Latex[length=3mm]}, thick] (v122) -- (v223);
    \draw [-{Latex[length=3mm]}, thick] (v123) -- (v221);
    \draw[fill=black] (v0) node[below=3pt,scale=1.2] {$f_2$};
    
    \draw [-{Latex[length=3mm]}, thick] (v111) arc [start angle=121.99, delta angle=-73.5, radius=5.03125];
    \draw [-{Latex[length=3mm]}, thick] (v112) arc [start angle=121.99, delta angle=-73.5, radius=5.03125];
    \draw [-{Latex[length=3mm]}, thick] (v113) arc [start angle=135.86, delta angle=-72.8, radius=5.125];
    
    \draw [-{Latex[length=3mm]}, thick] (v131) arc [start angle=228.49, delta angle=73.5, radius=5.03125];
    \draw [-{Latex[length=3mm]}, thick] (v132) arc [start angle=228.49, delta angle=73.5, radius=5.03125];
    \draw [-{Latex[length=3mm]}, thick] (v133) arc [start angle=243.06, delta angle=72.8, radius=5.125];
\end{tikzpicture}

\begin{tikzpicture}
    \coordinate (v11) at (-5,0);
    \coordinate (v21) at (-3,3.4);
    \coordinate (v31) at (-1,0);
    \coordinate (v41) at (-7,3.4);
    \coordinate (v51) at (1,3.4);
    \coordinate (v12) at (-5,-0.5);
    \coordinate (v22) at (-3,2.9);
    \coordinate (v32) at (-1,-0.5);
    \coordinate (v42) at (-7,2.9);
    \coordinate (v52) at (1,2.9);
    \coordinate (v13) at (-5,-1);
    \coordinate (v23) at (-3,2.4);
    \coordinate (v33) at (-1,-1);
    \coordinate (v43) at (-7,2.4);
    \coordinate (v53) at (1,2.4);
    \coordinate (v03) at (-3,4.4);
    
    \draw[fill=gray!30] (-5,-0.5) ellipse (8mm and 12mm);
    \draw[fill=gray!30] (-3,2.9) ellipse (8mm and 12mm);
    \draw[fill=gray!30] (-1,-0.5) ellipse (8mm and 12mm);
    \draw[fill=gray!30] (-7,2.9) ellipse (8mm and 12mm);
    \draw[fill=gray!30] (1,2.9) ellipse (8mm and 12mm);
    
    \draw[fill=black] (v11) circle[radius=1.5pt]; 
    \draw[fill=black] (v21) circle[radius=1.5pt]; 
    \draw[fill=black] (v31) circle[radius=1.5pt]; 
    \draw[fill=black] (v41) circle[radius=1.5pt]; 
    \draw[fill=black] (v51) circle[radius=1.5pt]; 
    \draw[fill=black] (v12) circle[radius=1.5pt]; 
    \draw[fill=black] (v22) circle[radius=1.5pt]; 
    \draw[fill=black] (v03) node[above=3pt,scale=1.2] {$\HH$};
    \draw[fill=black] (v32) circle[radius=1.5pt]; 
    \draw[fill=black] (v42) circle[radius=1.5pt]; 
    \draw[fill=black] (v52) circle[radius=1.5pt]; 
    \draw[fill=black] (v13) circle[radius=1.5pt]; 
    \draw[fill=black] (v23) circle[radius=1.5pt]; 
    \draw[fill=black] (v33) circle[radius=1.5pt]; 
    \draw[fill=black] (v43) circle[radius=1.5pt]; 
    \draw[fill=black] (v53) circle[radius=1.5pt]; 
    
    \draw (v11) -- (v21);
    \draw (v21) -- (v31);
    \draw (v31) -- (v11);
    \draw (v11) -- (v42);
    \draw (v21) -- (v43);
    \draw (v31) -- (v52);
    \draw (v12) -- (v22);
    \draw (v22) -- (v32);
    \draw (v32) -- (v12);
    \draw (v12) -- (v43);
    \draw (v22) -- (v41);
    \draw (v32) -- (v51);
    \draw (v13) -- (v23);
    \draw (v23) -- (v33);
    \draw (v33) -- (v13);
    \draw (v13) -- (v41);
    \draw (v23) -- (v42);
    \draw (v33) -- (v53);
\end{tikzpicture}
\end{center}

Before we prove Theorem~\ref{thm: upperboundFixed}, we need a lemma.

\begin{lem}\label{lem: upperGenFixed}
Suppose that $G_1,\ldots,G_n$ are vertex disjoint graphs where $n \geq 2$ and $G \in \bigoplus_{i=1}^{n} (G_i,p)$ where for each $i \in [n]$, $\{u_{i,1},\ldots,u_{i,p}\}$ is a clique in $G_i$ and $G$ is obtained by identifying $u_{1,q},\ldots,u_{n,q}$ as the same vertex $u_q$ for each $q \in [p]$. For each $i \in [n]$, suppose that $\HH_i = (L_i,H_i)$ is an $m$-fold cover of $G_i$ conducive to $\{u_{i,1},\ldots,u_{i,p}\}$. For each $2 \leq i \leq n$, suppose $f_i$ is a permutation of $[m]$. Let $F = (f_2,\ldots,f_n)$. Let $\gamma_1$ be the identity permutation of $[m]^p$, and let $\gamma_i : [m]^p \rightarrow [m]^p$ be the function defined by $\gamma_i((j_1,\ldots,j_p)) = (f_i(j_1),\ldots,f_i(j_p))$ for each $2 \leq i \leq n$. For any $\mathbf{j} = (j_1,\ldots,j_p) \in [m]^{p}$, let $P_{i,\mathbf{j}} = \{(u_{i,q},j_q) : q \in [p]\}$. Let $D = \sum_{\mathbf{j} \in [m]^p} N(P_{1,\mathbf{j}},\HH_1) \prod_{i=2}^{n} N(P_{i,\gamma_i(\mathbf{j})},\HH_i)$. Then $P_{DP}(G,\HH) = D$ where $\HH$ is the $F$-amalgamated cover of $G$ obtained from $\HH_1,\ldots,\HH_n$ which implies $P_{DP}'(G,\{u_{i,1},\ldots,u_{i,p}\},m) \leq D$.
\end{lem}
\begin{proof}
For each $\mathbf{j} = (j_1,\ldots,j_p) \in [m]^p$ such that $j_1,\ldots,j_p$ are pairwise distinct, let $P_{\mathbf{j}} = \{(u_q,j_q) : q \in [p]\}$. Let $\II_{i,\mathbf{j}}$ be the set of all $\HH_i$-colorings of $G_i$ that contain $P_{i,\mathbf{j}}$ for each $i \in [n]$, and $\II_{\mathbf{j}}$ be the set of all $\HH$-colorings of $G$ that contain $P_{\mathbf{j}}$. Let $g_{\mathbf{j}} : \prod_{i=1}^{n} \II_{i,\gamma_i(\mathbf{j})} \rightarrow \II_{\mathbf{j}}$ be given by $g_{\mathbf{j}}(I_{1,\gamma_1(\mathbf{j})},\ldots,I_{n,\gamma_n(\mathbf{j})})) = (\bigcup_{i=1}^{n} (I_{i,\gamma_i(\mathbf{j})} - P_{i,\gamma_i(\mathbf{j})})) \cup P_{\mathbf{j}}$.

Given a fixed element $\mathbf{j} = (j_1,\ldots,j_p)$ of $[m]^p$ such that $j_1,\ldots,j_p$ are pairwise distinct, we will now prove that $g_{\mathbf{j}}$ is a bijection. First, we will show that $g_{\mathbf{j}}$ is a function. Notice $|(\bigcup_{i=1}^{n} (I_{i,\gamma_i(\mathbf{j})} - P_{i,\gamma_i(\mathbf{j})})) \cup P_{\mathbf{j}}| = p + \sum_{i=1}^{n} (|V(G_i)|-p) = |V(G)|$. By the definition of $F$-amalgamated cover of $G$ obtained from $\HH_1,\ldots,\HH_n$, $vw \in E(H)$ for some $v,w \in I_{i,\gamma_i(\mathbf{j})} - P_{i,\gamma_i(\mathbf{j})}$ if and only if $vw \in E(H_i)$ for some $i \in [n]$. Since $I_{i,\gamma_i(\mathbf{j})}$ is an independent set in $H_i$ for each $i \in [n]$ and $V(G_1),\ldots,V(G_n)$ are pairwise disjoint, $\bigcup_{i=1}^{n} (I_{i,\gamma_i(\mathbf{j})} - P_{i,\gamma_i(\mathbf{j})})$ is an independent set in $H$.
By the definition of $F$-amalgamated cover of $G$ obtained from $\HH_1,\ldots,\HH_n$, $(u_q,j)(u_{q'},j') \in E(H)$ when $q \not= q'$ if and only if $j = j'$. Since $j_1,\ldots,j_p$ are pairwise distinct, $P_{\mathbf{j}}$ is an independent set in $H$.
By the definition of $F$-amalgamated cover of $G$ obtained from $\HH_1,\ldots,\HH_n$, $(u_q,j_q)v \in E(H)$ for some $(u_q,j_q) \in P_{\mathbf{j}}$ and $v \in I_{i,\gamma_i(\mathbf{j})} - P_{i,\gamma_i(\mathbf{j})}$ if and only if $(u_{1,q},j_q)v \in E(H_1)$ or $(u_{i+1,q},f_{i+1}(j_q))v \in E(H_i)$ for some $i \in [n-1]$. Since $I_{i,\gamma_i(\mathbf{j})}$ is an independent set in $H_i$ for each $i \in [n]$, $(\bigcup_{i=1}^{n} (I_{i,\gamma_i(\mathbf{j})} - P_{i,\gamma_i(\mathbf{j})})) \cup P_{\mathbf{j}}$ must be an independent set in $H$.

Now we will construct the inverse of $g_{\mathbf{j}}$, $h_{\mathbf{j}}$. Let $h_{\mathbf{j}} : \II_{\mathbf{j}} \rightarrow \prod_{i=1}^{n} \II_{i,\gamma_i(\mathbf{j})}$ be given by $h_{\mathbf{j}}(I_{\mathbf{j}}) =  (x_1,\ldots,x_n)$ where $x_i = (V(H_i) \cap I_{\mathbf{j}}) \cup P_{i,\gamma_i(\mathbf{j})}$ for each $i \in [n]$. First, we will show that $h_{\mathbf{j}}$ is a function. Clearly, $x_1 = (V(H_1) \cap I_{\mathbf{j}}) \cup P_{1,\mathbf{j}} \in \II_{1,\gamma_1(\mathbf{j})}$. Now we will show $x_i \in \II_{i,\gamma_i(\mathbf{j})}$ for each $2 \leq i \leq n$. By the definition of $F$-amalgamated cover of $G$ obtained from $\HH_1,\ldots,\HH_n$, because $I_{\mathbf{j}}$ is an independent set in $H$, $V(H_i) \cap I_{\mathbf{j}}$ is an independent set in $H_i$. Since $\HH_i$ is conducive to $P_{i,\mathbf{j}}$, $(u_{i,q},j)(u_{i,q'},j') \in E(H_i)$ when $q \not= q'$ if and only if $j = j'$. Since $j_1,\ldots,j_p$ are pairwise distinct, $f_i(j_1),\ldots,f_i(j_n)$ are pairwise distinct. Thus, $P_{i,\gamma_i(\mathbf{j})}$ is an independent set in $H_i$. By the definition of $F$-amalgamated cover of $G$ obtained from $\HH_1,\ldots,\HH_n$, $(u_{i,q},f_i(j_q))v \in E(H_i)$ for some $v \in V(H_i) \cap I_{\mathbf{j}}$ if and only if $(u_q,j_q)v \in E(H)$. Notice $(u_q,j_q)v \notin E(H)$ because $(u_q,j_q),v \in I_{\mathbf{j}}$. Thus, $x_i = (V(H_i) \cap I_{\mathbf{j}}) \cup P_{i,\gamma_i(\mathbf{j})} \in \II_{i,\gamma_i(\mathbf{j})}$.

Finally, we will show that $g_{\mathbf{j}}$ and $h_{\mathbf{j}}$ are inverses. We calculate
\begin{align*}
    g_{\mathbf{j}}(h_{\mathbf{j}}(I_{\mathbf{j}}))
    &= g_{\mathbf{j}}((x_1,\ldots,x_n))\\
    &= \left(\bigcup_{i=1}^{n} (((V(H_i) \cap I_{\mathbf{j}}) \cup P_{i,\gamma_i(\mathbf{j})}) - P_{i,\gamma_i(\mathbf{j})})\right) \cup P_{\mathbf{j}}\\
    &= \left(\bigcup_{i=1}^{n} V(H_i) \cap I_{\mathbf{j}} \right) \cup P_{\mathbf{j}}
    = I_{\mathbf{j}}.
\end{align*}
We calculate
\begin{align*}
    h_{\mathbf{j}}(g_{\mathbf{j}}((I_{1,\gamma_1(\mathbf{j})},\ldots,I_{n,\gamma_n(\mathbf{j})}))) &= h_{\mathbf{j}}\left(\left(\bigcup_{i=1}^{n} (I_{i,\gamma_i(\mathbf{j})} - P_{i,\gamma_i(\mathbf{j})})\right) \cup P_{\mathbf{j}}\right)
    = (y_1,\ldots,y_n)
\end{align*}
where for each $l \in [n]$,
\begin{align*}
    y_l &= \left(\left(\bigcup_{i=1}^{n} (I_{i,\gamma_i(\mathbf{j})} - P_{i,\gamma_i(\mathbf{j})}) \cup P_{\mathbf{j}}\right) \cap V(H_l)\right) \cup P_{l,\gamma_l(\mathbf{j})}\\
    &= \left((I_{l,\gamma_l(\mathbf{j})} - P_{l,\gamma_l(\mathbf{j})}) \cap V(H_l)\right) \cup P_{l,\gamma_l(\mathbf{j})}\\
    &= \left((I_{l,\gamma_l(\mathbf{j})} - P_{l,\gamma_l(\mathbf{j})}) \cup P_{l,\gamma_l(\mathbf{j})}\right) \cap V(H_l)
    = I_{l,\gamma_l(\mathbf{j})}.
\end{align*}
Therefore, $g_{\mathbf{j}}$ is bijection. As such,  \\ $N(P_{\mathbf{j}},\HH) = |\II_{\mathbf{j}}| = \prod_{i=1}^{n} |\II_{i,\gamma_i(\mathbf{j})}| = N(P_{1,\mathbf{j}},\HH_1) \prod_{i=2}^{n} N(P_{i,\gamma_i(\mathbf{j})},\HH_i)$ for each $\mathbf{j} = (j_1,\ldots,j_p) \in [m]^p$ such that $j_1,\ldots,j_p$ are pairwise distinct. Note $P_{DP}(G,\HH) = \sum_{\mathbf{j} \in [m]^p} N(P_{\mathbf{j}},\HH)$. Also note that $N(P_{\mathbf{j}},\HH) = 0$ when any two distinct coordinates of $\mathbf{j}$ are equal. Therefore, $P_{DP}(G,\HH) = D$ which implies that $P_{DP}'(G,\{u_{i,1},\ldots,u_{i,p}\},m) \leq D$.
\end{proof}

We are now ready to prove Theorem~\ref{thm: upperboundFixed}.

\begin{proof}
The proof is by induction on $n$. We begin by proving the result for $n = 2$. Let $X = \{u_1,\ldots,u_p\}$, and let $X_i = \{u_{i,1},\ldots,u_{i,p}\}$ for each $i \in [2]$. Suppose $\HH_i = (L_i,H_i)$ is an $m$-fold cover of $G_i$ conducive to $X_i$ such that $P_{DP}(G_i,\HH_i) = P_{DP}'(G_i,X_i,m)$.  For every $\mathbf{j} = (j_1,\ldots,j_p) \in [m]^p$, let $P_{i,\mathbf{j}} = \{(u_{i,q},j_q) : q \in [p]\}$ for each $i \in [2]$. Let $f_{\sigma} : [m]^p \rightarrow [m]^p$ be the function defined by $f_{\sigma}((j_1,\ldots,j_p)) = (\sigma(j_1),\ldots,\sigma(j_p))$ where $\sigma \in S_m$ and $S_m$ is the symmetric group on $[m]$. Notice $f_{\sigma}$ is a permutation.

Suppose $T$ is the set of all $(j_1',\ldots,j_p') \in [m]^p$ such that $j_1',\ldots,j_p'$ are pairwise distinct. Notice for each $(j_1',\ldots,j_p'),(j_1'',\ldots,j_p'') \in T$, $|\{\sigma \in S_m : f_{\sigma}((j_1',\ldots,j_p')) = (j_1'',\ldots,j_p'')\}| = (m-p)!$. So, for each $\mathbf{j'} \in T$ and $\sigma' \in S_m$, $\sum_{\sigma \in S_m} N(P_{2,f_{\sigma}(\mathbf{j'})},\HH_2) = (m-p)!\sum_{\mathbf{j} \in T} N(P_{2,f_{\sigma'}(\mathbf{j})},\HH_2)$. Notice for each $i \in [2]$, $N(P_{i,\mathbf{j}},\HH_i) = 0$ when any two distinct coordinates of $\mathbf{j}$ are equal which implies that $\sum_{\mathbf{j} \in T} N(P_{2,f_{\sigma'}(\mathbf{j})},\HH_2) = \sum_{\mathbf{j} \in [m]^p} N(P_{2,f_{\sigma'}(\mathbf{j})},\HH_2)$. Notice $P_{DP}'(G_1,X_1,m) = \sum_{\mathbf{j} \in [m]^p} N(P_{1,\mathbf{j}},\HH_1)$ and $P_{DP}'(G_2,X_2,m) = \sum_{\mathbf{j} \in [m]^p} N(P_{2,f_{\sigma}(\mathbf{j})},\HH_2)$ for each $\sigma \in S_m$. Let $D_\sigma = \sum_{\mathbf{j} \in [m]^p} N(P_{1,\mathbf{j}},\HH_1) N(P_{2,f_{\sigma}(\mathbf{j})},\HH_2)$. We calculate that for any $\sigma' \in S_m$c,
\begin{align*}
    \sum_{\sigma \in S_m} D_\sigma
    &= \sum_{\sigma \in S_m} \sum_{\mathbf{j} \in [m]^p} N(P_{1,\mathbf{j}},\HH_1) N(P_{2,f_{\sigma}(\mathbf{j})},\HH_2)\\
    &= \sum_{\mathbf{j} \in [m]^p} \left(N(P_{1,\mathbf{j}},\HH_1) \sum_{\sigma \in S_m} N(P_{2,f_{\sigma}(\mathbf{j})},\HH_2)\right)\\
    &= \sum_{\mathbf{j} \in [m]^p} \left(N(P_{1,\mathbf{j}},\HH_1) (m-p)!\sum_{\mathbf{j} \in [m]^p} N(P_{2,f_{\sigma'}(\mathbf{j})},\HH_2)\right)\\
    &= (m-p)!P_{DP}'(G_2,X_2,m)\sum_{\mathbf{j} \in [m]^p} N(P_{1,\mathbf{j}},\HH_1)\\
    &= (m-p)!P_{DP}'(G_2,X_2,m)P_{DP}'(G_1,X_1,m).
\end{align*}
By Lemma~\ref{lem: upperGenFixed}, we know that $P_{DP}'(G,X,m) \leq D_\sigma$ for each $\sigma \in S_m$. Thus, there is a $\sigma' \in S_m$ such that $$P_{DP}'(G,X,m) \leq D_{\sigma'} \leq \frac{\sum_{\sigma \in S_m} D_\sigma}{m!} = \frac{P_{DP}'(G_1,X_1,m)P_{DP}'(G_2,X_2,m)}{\prod_{i=0}^{p-1} (m-i)}$$
which completes the basis step since $P_{DP}(G,m) \leq P_{DP}'(G,X,m)$. 

Now suppose $n \geq 3$ and the result holds for all natural numbers greater than $1$ and less than $n$. Let $G'$ be the graph obtained by identifying $u_{1,q},\ldots,u_{n-1,q}$ as the same vertex $u_q'$ for each $q \in [p]$. Let $X' = \{u_1',\ldots,u_p'\}$. By the inductive hypothesis, for each $m \geq p$, $$P_{DP}'(G',X',m) \leq \frac{\prod_{i=1}^{n-1} P_{DP}'(G_i,X_i,m)}{\prod_{i=0}^{p-1} (m-i)^{n-2}}.$$ Notice $G$ is the graph obtained by identifying $u_q'$ and $u_{n,q}$ as the same vertex $u_q$ for each $q \in [p]$. Thus, by the inductive hypothesis, $$P_{DP}'(G,X,m) \leq \frac{P_{DP}'(G',X',m)P_{DP}'(G_n,X_n,m)}{\prod_{i=0}^{p-1} (m-i)} \leq \frac{\prod_{i=1}^{n} P_{DP}'(G_i,X_i,m)}{\prod_{i=0}^{p-1} (m-i)^{n-1}}$$ which implies that $P_{DP}(G,m) \leq \left(\prod_{i=1}^{n} P_{DP}'(G_i,X_i,m)\right)/\left(\prod_{i=0}^{p-1}(m-i)^{n-1}\right)$.
\end{proof}

{\bf Acknowledgment.}  This paper is based on a research project conducted with undergraduate students Michael Maxfield and Seth Thomason at the College of Lake County during the spring and summer of 2021. The support of the College of Lake County is gratefully acknowledged. The authors also thank Vu Bui for helpful conversations.

\end{document}